\documentclass[11pt]{article}

\usepackage{amsmath,amsthm}

\usepackage{amssymb,latexsym}

\usepackage[mathscr]{euscript}

\usepackage{enumerate}

\def\toP{{\scriptstyle\,\buildrel{\P}\over{\hbox to
0.6cm{\rightarrowfill}}\,}}

\newcommand{\BR}{{\mathbb R}}

\newcommand{\sgn}{\mbox{\rm sgn}}

\newcommand{\TT}{{\cal T}}

\newcommand{\MM}{{\cal M}}
\newcommand{\VV}{{\cal V}}
\newcommand{\EE}{{\cal E}}

\newcommand{\FF}{{\cal F}}

\newcommand{\essinf}{\mathop{\mathrm{ess\,inf}}}
\newcommand{\esssup}{\mathop{\mathrm{ess\,sup}}}

\newtheorem{theorem}{\bf Theorem}[subsection]
\newtheorem{proposition}[theorem]{\bf Proposition}%[subsection]
\newtheorem{lemma}[theorem]{\bf Lemma}%[subsection]

\theoremstyle{definition}
\newtheorem{definition}[theorem]{Definition}
\newtheorem{remark}[theorem]{\bf Remark}%[subsection]

\newcommand{\nsubsection}{\setcounter{equation}{0}\subsection}

\begin{document}

\title{Reflected BSDEs with two optional barriers and monotone coefficient on general filtered space}
\author {Tomasz Klimsiak and Maurycy Rzymowski\\
{\small Faculty of Mathematics and Computer Science,
Nicolaus Copernicus University} \\
{\small  Chopina 12/18, 87--100 Toru\'n, Poland}}%\\
%{\small E-mail addresses: tomas@mat.umk.pl (T. Klimsiak), rozkosz@mat.umk.pl (A. Rozkosz)}}
%\author {Tomasz Klimsiak and Andrzej Rozkosz}
\date{}
\maketitle
\begin{abstract}
We consider reflected backward stochastic differential equations with two optional barriers of class (D) satisfying Mokobodzki's separation
condition and coefficient which is only continuous and non-increasing.
%with respect to value variable.
We  assume that data are merely integrable and the terminal time is an arbitrary (possibly infinite) stopping time. We study the problem of  existence and uniqueness of solutions, and their connections with the value process in  nonlinear Dynkin games.
\end{abstract}



\nsubsection{Introduction}

Let $\mathbb{F}=(\FF_t)_{t\ge 0}$ be a general filtration satisfying merely the usual conditions and   $T$ be an arbitrary (possibly infinite) $\mathbb{F}$-stopping time. We also assume as given an $\FF_T$-measurable random variable $\xi$,  a real function $f$ defined on $\Omega\times \BR^+\times \BR$, which is  $\mathbb F$-progressively measurable with respect to first two  variables, and two $\mathbb{F}$-optional processes $L,U$ satisfying some separation condition.
In the present paper, we consider reflected backward stochastic differential equations (RBSDE for short)  which informally can be written in the form
\begin{equation}
\label{intr1}
\begin{cases}
dY_t=- f(t,Y_t)\,dt-\,dR_t+\,dM_t\,\,\, \mbox{on}\,\,\,[0,T],\quad Y_T=\xi,\\
(Y,M,R)\in \mathcal O\times \MM_{loc}\times \VV_p\,, \\
L\le Y\le U\,\,\,\mbox{on}\,\,\, [0,T],\quad  dR\,\, \mbox{is ``minimal"}.
\end{cases}
\end{equation}
In (\ref{intr1}), $ \mathcal O$ (resp. $\MM_{loc}, \VV_p$) is the space of $\mathbb F$-optional processes (resp. local $\mathbb F$-martingales, finite variation $\mathbb F$-predictable processes).
We study the problem of existence and uniqueness of solutions to (\ref{intr1}) and connections of (\ref{intr1}) with nonlinear Dynkin games.

The notion of BSDEs with two reflecting barriers was introduced in 1996 by Cvitanic and Karatzas \cite{CK}. They considered bounded terminal time $T$,  continuous barriers $L,U$ and  filtration $\mathbb F$ generated by a Brownian motion. In that case (\ref{intr1}) may be formulated rigorously as follows:
\begin{equation}
\label{eq1.1}
\begin{cases}
Y_t=\xi+\int^T_t f(r,Y_r,Z_r)\,dr+\int^T_t\,dR_r-\int^T_t Z_r\,dB_r,\quad t\in[0,T],\\
(Y,Z,R)\in \mathcal O_c\times \mathcal O\times \VV_c,\\
L_t\le Y_t\le U_t,\,\, t\in[0,T],\quad \int^T_0(Y_r-L_r)\,dR^+_r=\int^T_0(U_r-Y_r)\,dR^-_r=0,
\end{cases}
\end{equation}
where $\mathcal O_c$ (resp. $\mathcal V_c$) is the subspace of $\mathcal O$ (resp. $\mathcal V_p$) consisting of continuous
processes.
In \cite{CK} it is assumed that the generator $f$ is Lipschitz continuous with respect to $y,z$ and the data are in $L^2$, that is  $\xi, f(\cdot,0,0)$, $\sup_{t\le T} |L_t|$ and $\sup_{t\le T}|U_t|$ are square-integrable. Note that in the special case of Brownian filtration each  local $\mathbb F$-martingale $M$ is of the form $M=Z\,dB$, which allows one to  consider $f$ depending also on $Z$.

In \cite{CK} an existence and uniqueness result for (\ref{eq1.1}) is proved. Moreover, it is shown there that the solution is linked with Dynkin game via  the  formula
\begin{equation}\label{intr2}
Y_\alpha=\esssup_{\sigma\ge \alpha}\essinf_{\tau\ge \alpha}E\Big(\int_\alpha^{\tau\wedge\sigma}f(r,Y_r,Z_r)\,dr+L_\sigma\mathbf{1}_{\sigma<\tau}
+U_\tau\mathbf{1}_{\tau\le\sigma<T}+\xi\mathbf{1}_{\sigma=\tau=T}|\FF_{\alpha}\Big)
\end{equation}
holding for any stopping time $\alpha\in[0,T]$. A similar result was obtained in 1997 by El Karoui et al. \cite{EKPPQ} in the case of one barrier.

The theory of RBSDEs has been significantly developed over the last two decades and assumptions from the  paper by  Cvitanic and Karatzas \cite{CK} were successively
weakened. We will provide a  brief review of the literature to show the main directions of relaxing of the standard assumptions considered 
in the pioneering paper by Cvitanic and Karatzas.   The case of c\`adl\`ag barriers  is considered in \cite{Ha1,KBSM,LX}. RBSDEs with   monotone generator satisfying weak growth condition are studied in \cite{KBSM,LMX,RS} and we refer the reader to \cite{BY,HP,KBSM,RS} for equations with $L^p$-data for $p\in[1,2]$. Equations with  Brownian-Poisson filtration and c\`adl\`ag barriers were studied in \cite{HH1,HO,HO1,HW,RY1}. The case of general, right-continuous filtration $\mathbb{F}$, monotone generator $f$ and $L^1$-data was studied in \cite{KSPA}. Equations with  $T$ being an arbitrary stopping time were studied in \cite{RS1},  in \cite{AO} (Brownian-Poisson filtration) and  \cite{Kl} (general filtration).

In most of the existing papers on RBSDEs c\`adl\`ag barriers are considered, and
there are only few papers dealing with non-c\`adl\`ag case. Such equations with
$L^2$-data and Lipschitz continuous generator were studied in  \cite{MM} (Brownian filtration), in \cite{GIOOQ,GIOQ2,MM2} (Brownian-Poisson filtration) and
\cite{BO,BO2,GIOQ} (general filtration).  RBSDEs with $L^1$-data and optional barriers
were considered only in \cite{KRzS,KRzS2} in  case of Brownian filtration and bounded terminal time.

As already mentioned, in the present paper we study the existence and uniqueness of solutions of class (D) to RBSDEs (\ref{intr1}) with  general filtration  $\mathbb{F}$ and possibly infinite $\mathbb{F}$-stopping time $T$. We assume  that $L$ and $U$ are $\mathbb{F}$-optional processes of class (D) satisfying Mokobodzki's condition  and such that
\begin{equation}
\label{intr13}
\limsup_{a\rightarrow\infty}L_{T\wedge a}\le\xi
\le\liminf_{a\rightarrow\infty}U_{T\wedge a}.
\end{equation}
As for  $\xi,f$ and $V$ we will assume that they  satisfies the following conditions.
\begin{enumerate}
\item[(H1)] $E(|\xi|+\int^T_0\,d|V|_r+\int^T_0|f(r,0)|\,dr)<\infty$,
\item[(H2)] for a.e. $t\in[0,T]$ the function $t\mapsto f(t,y)$ is nonincreasing,
\item[(H3)] $y\mapsto f(t,y)$ is continuous for every $t\in[0,T]$,
\item[(H4)] $\int^T_0|f(r,y)|\,dr<\infty$ for all $y\in\mathbb{R}$.
\end{enumerate}

In Section \ref{sec3}, for $f$ satisfying (H1)--(H4) we  introduce the notion  of nonlinear
$f$-expectation
\[
\mathcal{E}^f_{\alpha,\beta}:L^1(\Omega,\mathcal{F}_{\beta},P)\rightarrow L^1(\Omega,\mathcal{F}_{\alpha},P),
\]
associated with BSDE (\ref{intr1}) with no reflection, and we prove its basic properties. Here $\alpha\le \beta$ are stopping times.

In Section \ref{sec4}, we give a definition of a solution to (\ref{intr1}).
By the definition, we are looking for a triple $(Y,M,R)\in\mathcal O\times \MM_{loc}\times \VV_p$  such that $Y$ is a regulated (l\`adl\`ag) process such that
\begin{align*}
Y_t=Y_{T\wedge a}+\int^{T\wedge a}_t f(r,Y_r)\,dr+\int_{t}^{T\wedge a}\,dR-\int^{T\wedge a}_t\,dM_r,\quad t\in[0,T\wedge a].
\end{align*}
for every $a\ge0$, and moreover, $Y$ satisfies the terminal condition of the form
\[
Y_{T\wedge a}\rightarrow \xi,\quad a\rightarrow \infty.
\]
We also require that $R$ is minimal in the sense that for every $a\ge0$,
\begin{align*}
&\int^{T\wedge a}_0(Y_{r-}-\overrightarrow{L}_r)\,dR^{*,+}_r+\sum_{0\le r<T\wedge a}(Y_r-L_r)\Delta^+R^+_r\\
&\qquad+\int^{T\wedge a}_0(\underrightarrow{U}_r-Y_{r-})\,dR^{*,-}_r+\sum_{0\le r<T\wedge a}(U_r-Y_r)\Delta^+R^-_r=0,
\end{align*}
where $R=R^+-R^-$ is the Jordan decomposition of $R$,
\[
\overrightarrow{L}_r=\limsup_{s\uparrow r}X_s,\qquad \underrightarrow{U}_r=\liminf_{s\uparrow r}X_s,
\]
and $R^{*,+}$ (resp. $R^{*,-}$) is the c\`adl\`ag part of the process $R^{+}$ (resp. $R^{-}$).
We prove that there exists at most one solution $(Y,M,R)$ to RBSDE (\ref{intr1}) such that $Y$ is of class (D).

In Section \ref{sec5}, we prove that under (H1)--(H4) there exists a solution $(Y,M,R)$ to RBSDE (\ref{intr1}) with one reflecting lower barrier $L$ such that $Y$ is of class (D). In this case $R$ is an increasing process.
We also show that for every stopping time $\alpha\le T$,
\[
Y_\alpha=\esssup_{\tau\ge \alpha}\EE^f_{\alpha,\tau}(L_\tau\mathbf{1}_{\tau<T}+\xi\mathbf{1}_{\tau=T}).
\]
Let us stress here that in general $ER_T$ and $E\int_0^T|f(r,Y_r)|\,dr$ are
infinite. We give necessary and sufficient condition
for which these integrals are finite. We show that if this is the case, then
\[
Y_\alpha=\esssup_{\tau\ge \alpha}E(\int_\alpha^\tau f(r,Y_r)\,dr +L_\tau\mathbf{1}_{\tau<T}+\xi\mathbf{1}_{\tau=T}|\FF_\alpha).
\]

In Section \ref{sec6}, we are focused on our main goal, i.e.  equation (\ref{intr1}).
We first show that each solution to (\ref{intr1}) such that $Y$ is of class (D) admits the representation
\[
Y_{\alpha}=\esssup_{\rho=(\tau,H)\in\mathcal{S}_{\alpha}}\essinf_{\delta=(\sigma,G)\in\mathcal{S}_{\alpha}}\mathcal{E}^f_{\alpha,\tau\wedge\sigma}(L^u_{\rho}\mathbf{1}_{\{\tau\le\sigma<T\}}+U^l_{\delta}\mathbf{1}_{\{\sigma<\tau\}}+\xi\mathbf{1}_{\{\tau=\sigma=T\}}),
\]
where $\rho,\delta$ are the so called \textit{stopping systems} (see Section \ref{sec6}), and
\[
L^u_{\rho}=L_{\tau}\mathbf{1}_H+\overleftarrow{L}_{\tau}\mathbf{1}_{H^c},\qquad U^l_{\rho}=U_{\tau}\mathbf{1}_G+\underleftarrow{U}_{\tau}\mathbf{1}_{G^c}.
\]
Using this representation we prove a stability result for (\ref{intr1}).
To prove the existence of a solution,  we consider the  nonlinear
decoupling system
\begin{equation}\label{intr3}
\begin{cases}
Y^1_t=\esssup_{t\le\tau\le T}E(Y^2_\tau+\int_t^\tau f(r,Y^1_r-Y^2_r)\,dr+L_\tau\mathbf{1}_{\tau<T}+\xi\mathbf{1}_{\tau=T}|\FF_t),\\
 Y^2_t=\esssup_{t\le\tau\le T}E(Y^1_\tau\mathbf{1}_{\tau<T}-U_\tau\mathbf{1}_{\tau<T}|\FF_t)
\end{cases}
\end{equation}
introduced in the linear case ($f\equiv 0$)  by Bismut \cite{Bismut1}. Since, as we mentioned before, the integral in (\ref{intr3}) may be infinite,
we reformulate (\ref{intr3}) as a the following system of RBSDEs with one reflecting lower barrier:
\begin{equation}\label{intr12}
\begin{cases}
dY^1_t=- f(t,Y^1_t-Y^2_t)\,dr-\,dR^1_t+\,dM^1_t\,\,\, \mbox{on}\,\,\,[0,T],\quad Y^1_T=\xi,\\
dY^2_t=-\,dR^2_t+\,dM^2_t\,\,\, \mbox{on}\,\,\,[0,T],\quad Y^2_T=0,\\
Y^2+L\le Y^1,\quad Y^1-U\le Y^2.
\end{cases}
\end{equation}
Using  the results of Section \ref{sec5} we prove that  under (H1)--(H4), (\ref{intr13}) and Mokobodzki's condition (the existence of a special semimartingale between the barriers) there exists a solution $(Y^1, M^1,R^1)$, $(Y^2,M^2,R^2)$ to (\ref{intr12}) such that $Y^1, Y^2$ are of class (D). Next we show that the triple
 \[
 (Y,M,R):= (Y^1, M^1,R^1)-(Y^2,M^2,R^2)
 \]
is a solution to (\ref{intr1}).  We also give a necessary and sufficient condition under which $E|R|_T$ and $E\int_0^T|f(r,Y_r)|\,dr$ are finite. Finally, using our stability result, we show that there exists a solution $(Y,M,R)$ to (\ref{intr1}) such that $Y$ is of class (D) even if $f$ does not satisfy (H1),  i.e. under  (H2)--(H4), (\ref{intr13}) and Mokobodzki's condition.  It is worth pointing here that in general,  without (H1) imposed on $f$ there is no solution to equation of  type (\ref{intr1})  with no reflection. In other words, we show that for the existence of solutions to reflected BSDEs weaker assumptions on $f$ are needed  than for the existence of solutions to related BSDEs.

\nsubsection{Notation and standing assumptions}
\label{sec2}

Let $a\ge 0$. We say that a function $y:[0,a]\rightarrow\mathbb{R}$ is  regulated if the limit $y_{t+}=\lim_{u\downarrow t}y_u$ exists for every $t\in[0,a)$, and the limit $y_{s-}=\lim_{u\uparrow s}y_u$ exists for every $s\in(0,a]$.
For any regulated function $y$ on $[0,a]$ we set $\Delta^{+}y_t=y_{t+}-y_t$ if $0\leq t< a$, and  $\Delta^{-}y_s=y_s-y_{s-}$ if  $0<s\leq a$. It is known that each regulated function is bounded and has at most countably many discontinuities (see, e.g., \cite[Chapter 2, Corollary 2.2]{dn}).
For $x\in\mathbb{R}$, we set $\sgn(x)=\mathbf{1}_{x\neq 0}\,x/|x|$.

Let $(\Omega,\mathcal{F},P)$ be a probability space, $\mathbb{F}=(\mathcal{F}_t)_{t\ge 0}$ be a filtration satisfying the usual conditions and let $T$ be an $\mathbb{F}$-stopping time. For $a\ge 0$, we set $T_a=T\wedge a$.
For  fixed stopping times $\sigma,\tau$ we denote by $\mathcal{T}_{\sigma,\tau}$ the set of all $\mathbb{F}$- stopping times taking values in $[\sigma,\tau]$. We alsoset $\mathcal{T}=\mathcal{T}_{0,T}$ and $\mathcal{T}_{\tau}=\mathcal{T}_{\tau,T}$. An increasing sequence $\{\tau_k\}\subset\mathcal{T}$ is called a chain (on $[0,T]$) if
\[
\forall{\omega\in\Omega}\quad
\exists{n\in\mathbb{N}}\quad\forall{k\ge n}\quad\tau_k(\omega)=T.
\]

We say that an $\mathbb{F}$-progressively measurable process $X$ is of class (D) if the family $\{X_{\tau},\,\tau\in\mathcal{T},\,\tau<\infty\}$ is uniformly integrable. We equip the space of processes of class (D) with the norm $\|X\|_1=\sup_{\tau\in\mathcal{T},\,\tau<\infty}E|X_{\tau}|$. In the sequel  in case if $X_\infty$ is not defined we set $X_\infty=0$.

We denote by $\mathcal{M}$ (resp. $\mathcal{M}_{loc}$) the set of all $\mathbb{F}$- martingales (resp. local martingales) such that $M_0=0$, and by $\mathcal{V}$ (resp. $\mathcal{V}^+$) the space of all $\mathbb{F}$- progressively measurable processes of finite variation (resp. increasing) such that $V_0=0$. $\mathcal{V}^1$ is the set of processes $V\in\mathcal{V}$ (resp. $V\in\mathcal{V}^+$) such that $E|V|_T<\infty$, where $|V|_T$ stands for the total variation of $V$ on $[0,T]$. $\mathcal{V}_p$ (resp. $\mathcal{V}^+_p$) is the space of all predictable $V\in\mathcal{V}$ (resp. $V\in\mathcal{V}^+$).  For $V\in \mathcal{V}$, by $V^*$ we denote the
c\`adl\`ag part of the process $V$, and by  $V^d$ its purely
jumping part consisting of right jumps, i.e.
\[
V^d_t=\sum_{s<t}\Delta^+V_s,\quad V^*_t=V_t-V^d_t,\quad t\in
[0,T].
\]

In the whole paper all relations between random variables are understood to hold
$P$-a.s. For processes $X$ and $Y$, we  write $X\le Y$ if $X_t\le Y_t$,
$t\in[0,T_a]$, $a\ge 0$. For a process $X$, we set $\overrightarrow{X}_s=\limsup_{r\uparrow s}X_r$, $\overleftarrow{X}_s=\limsup_{r\downarrow s}X_r$, $\underrightarrow{X}_s=\liminf_{r\uparrow s}X_r$ and $\underleftarrow{X}_s=\liminf_{r\downarrow s}X_r$, $s\in[0,T_a]$, $a\ge 0$. By \cite[Theorem 90, page 143]{dm}, if $X$ is an optional process  of class (D), then $\overleftarrow X, \underleftarrow X$ are progressively measurable, and $\overrightarrow X, \underrightarrow X$ are predictable processes.

In the whole paper,  $L$ and $U$ are $\mathbb{F}$-adapted optional processes of class (D), and $\xi$ is $\mathcal{F}_T$-measurable random variable. We always assume that (\ref{intr13}) is satisfied.
%\[
%\limsup_{a\rightarrow\infty}L_{T_a}\le\xi\le\liminf_{a\rightarrow\infty}U_{T_a}.
%\]
The generator (coefficient) is a map
\[
\Omega\times[0,T]\times\mathbb{R}\ni(\omega,t,y)\mapsto f(\omega,t,y)\in\mathbb{R},
\]
which is $\mathbb{F}$-progressively measurable for  fixed $y$. As for $V$, we
always assume that $V\in\VV$.

\nsubsection{BSDEs and nonlinear $f$-expectation}
\label{sec3}

\begin{definition}\label{pt0}
We say that a pair $(Y,M)$ of $\mathbb F$-adapted processes is a solution of the backward stochastic differential equation on the interval $[0,T]$ with right-hand side $f+dV$, terminal value $\xi$ (BSDE$^T$($\xi,f+dV$) for short) if
\begin{enumerate}
\item[(a)] $Y$ is regulated and $M\in\mathcal{M}_{loc}$,
\item[(b)] $\int^{T_a}_0|f(r,Y_r)|\,dr<+\infty$ for every $a\ge 0$,
\item[(c)] for every $a\ge 0$,
\begin{align*}
Y_t=Y_{T_a}+\int^{T_a}_t f(r,Y_r)\,dr+\int^{T_a}_t\,dV_r-\int^{T_a}_t\,dM_r,\quad t\in[0,T_a],
\end{align*}
\item[(d)] $\lim_{a\rightarrow\infty}Y_{T_a}=\xi$ a.s.
\end{enumerate}
\end{definition}

\begin{remark}
Existence, uniqueness and some other  properties of solutions to equation
BSDE$^T$($\xi,f+dV)$ we will use later on follow from \cite{Kl} (see also \cite{K7}).
In these papers it is assumed that $V$ is c\`adl\`ag but the results of \cite{Kl,K7} may
be applied to the case when $V$ is regulated by a simple change of variables. Indeed,
if  $(\bar{Y},\bar{M})$ is a solution to BSDE$^T(\xi,f^*+dV^*)$ with
\[
f^*(t,y)= f(t,y+V^d_t),
\]
then  $(Y,\bar{M})$ with  $Y=\bar{Y}+V^d$ is a solution to BSDE$^T(\xi,f+dV)$.
\end{remark}

\begin{lemma}
\label{lm.ad3.1}
Assume \mbox{\rm(H1), (H2)}. Let $(Y,M)$ be a solution to BSDE$^T(\xi,f)$. Then for every $\delta\in \TT$,
\[
E\Big(\int_\delta^T |f(r,Y_r)|\,dr\big|\FF_\delta\Big)\le 2 E\Big(|\xi|-|Y_\delta|+\int_\delta^T|f(r,0)|\,dr\big|\FF_\delta\Big).
\]
\end{lemma}
\begin{proof}
By the Meyer-Tanaka formula,
\[
|Y_\delta|\le |\xi|+\int_\delta^T \mbox{sgn}(Y_r)f(r,Y_r)\,dr-\int_\delta^T \mbox{sgn}(Y_{r-})\,dM_r.
\]
By (H2),
\[
-\mbox{sgn}(Y_r)(f(r,Y_r)-f(r,0))\ge 0,
\]
which when combined with the previous inequality and standard localization argument  gives the desired result.
\end{proof}

We now introduce the notion of nonlinear expectation.
\begin{definition}
Assume \textnormal{(H1)--(H4)}. Let $\alpha,\beta\in\mathcal{T}$, $\alpha\le\beta$. We say that an operator
\[
\mathcal{E}^f_{\alpha,\beta}:L^1(\Omega,\mathcal{F}_{\beta},P)\rightarrow L^1(\Omega,\mathcal{F}_{\alpha},P)
\]
is a nonlinear $f$-expectation if
\[
\mathcal{E}^f_{\alpha,\beta}(\xi)=Y_{\alpha},
\]
 where $(Y,M)$ is the unique solution of \textnormal{BSDE}$^{\beta}(\xi,f)$.
\end{definition}

We say that a c\`adl\`ag process $X$ of class (D) is an $\EE^f$-supermartingale
(resp. $\EE^f$-submartingale) on $[\alpha,\beta]$
if $\EE^f_{\sigma,\tau}(X_\tau)\le X_\sigma$ (resp. $\EE^f_{\sigma,\tau}(X_\tau)\ge X_\sigma$) for all
$\tau,\sigma\in \TT$ such that $\alpha\le\sigma\le\tau\le\beta$.
$X$ is called an $\EE^f$-martingale on $[\alpha,\beta]$ if it is both $\EE^f$-supermartingale and $\EE^f$-submartingale on $[\alpha,\beta]$.

For a given finite variation  process $V$ and stopping times $\alpha,\beta$ ($\alpha\le \beta$) we denote  by $|V|_{\alpha,\beta}$ the total variation of the process $V$ on $[\alpha,\beta]$. By BSDE$^{\alpha,\beta}(\xi,f+dV)$ we denote the problem (a)--(d) of Definition \ref{pt0}
but with $0,T$ replaced by $\alpha,\beta$, respectively.

In Proposition \ref{stw.fexp} we gather some properties of nonlinear $f$-expectation which will be needed later on. These properties are direct consequences of  the existence, uniqueness, stability,  and comparison
theorems for BSDEs proved in  \cite{Kl}. Properties (i)-(iii)  below were proved in \cite[Proposition 5.6]{K7}.

\begin{proposition}
\label{stw.fexp}
Assume that  $f$ satisfies \mbox{\rm(H1)--(H4)}  and $\alpha,\beta\in\TT$, $ \alpha\le\beta$.
\begin{enumerate}
\item[\rm(i)] Let $\xi\in L^1(\Omega,\FF_\beta,P)$ and $V$ be an $\mathbb F$-adapted finite
variation process such that  $V_\alpha=0$ and $E|V|_{\alpha,\beta}<\infty$. Then there exists a unique solution $(X,N)$
of  BSDE$^{\alpha,\beta}(\xi,f+dV)$.
Moreover,  if $V$ (resp. $-V$) is an increasing process, then $X$ is an $\EE^f$-supermartingale (resp. $\EE^f$-submartingale) on $[\alpha,\beta]$.

\item[\rm(ii)] If  $\xi_1,\xi_2\in L^1(\Omega,\FF_\beta,P)$ and $\xi_1\le\xi_2$, then $\EE^f_{\alpha,\beta}(\xi_1)\le\EE^f_{\alpha,\beta}(\xi_2)$.

\item[\rm(iii)] If $f_1,f_2$ satisfy \mbox{\rm(H1)--(H4)} with $\mu\le 0$, $\alpha,\beta_1,\beta_2\in\TT$, $\alpha\le\beta_1\le\beta_2$, $\xi_1\in L^1(\Omega,\FF_{\beta_1},P)$, $\xi_2\in L^1(\Omega,\FF_{\beta_2},P)$, then
\begin{align*}
|\EE^{f_1}_{\alpha,\beta_1}(\xi_1)-\EE^{f_2}_{\alpha,\beta_2}(\xi_2)|&\le E\Big(|\xi_1-\xi_2|+\int_\alpha^{\beta_1}|f^1(r,Y^1_r)-f^2(r,Y^1_r)|\,dr \\ &\quad+\int_{\beta_1}^{\beta_2}|f^2(r,Y^2_r)|\,dr|\FF_\alpha\Big),
\end{align*}
where $Y^1_t=\EE^{f^1}_{t\wedge\beta_1,\beta_1}(\xi_1)$, $Y^2_t=\EE^{f^2}_{t\wedge\beta_2,\beta_2}(\xi_2)$,

\item[\rm{(iv)}] For every $A\in \FF_\alpha$,
\[
\mathbf{1}_A\EE^f_{\alpha,\beta}(\xi)=\EE^{f_A}_{\alpha,\beta}(\mathbf{1}_A\xi),
\]
where $f_A(t,y)=f(t,y)\mathbf{1}_A\mathbf{1}_{t\ge \alpha}$,

\item[\rm{(v)}] For every $\gamma\in \TT$ such that $\gamma\ge\beta$,
\[
\EE^f_{\alpha,\beta}(\xi)=\EE^{f^\beta}_{\alpha,\gamma}(\xi),
\]
where $f^\beta(t,y)=f(t,y)\mathbf{1}_{t\le\beta}$.
\end{enumerate}
\end{proposition}

We will also need the following lemma.

\begin{lemma}
\label{lm.ad3.ilace}
Assume \mbox{\rm(H1)--(H4)}. Let $X$ be an $\mathbb F$-adapted process of class \mbox{\rm(D)} on $[0,T]$.
Then there exists a supermartingale $U$ of class \mbox{\rm(D)} on $[0,T]$ such that
\[
|\EE^f_{\alpha,\tau}(X_\tau)|\le U_\alpha,\quad \alpha\in\TT.
\]
Moreover, if $\{\delta_n\}\subset \TT$
is an increasing sequence such that $\delta_n\nearrow T$, then
\[
\lim_{n\rightarrow \infty}\sup_{\tau\in\TT}E\int_{\tau\wedge \delta_n}^\tau|f(r,\EE^f_{r,\tau}(X_\tau))|\,dr=0.
\]
\end{lemma}
\begin{proof}
Let $\varepsilon>0$.
By Lemma \ref{lm.ad3.1}, for all $\alpha,\tau\in \TT$ such that $\alpha\le\tau$,
\begin{align}
\label{eq.ad3.2}
\nonumber|\EE^f_{\alpha,\tau}(X_\tau)|&\le 2E\big(|X_\tau|+\int_\alpha^\tau|f(r,0)|\,dr\big|\FF_\alpha\big)\\
&\le
2\esssup_{\tau\in \TT_{\alpha,T}}E\big(|X_\tau| \big|\FF_\alpha\big)
+2E\big(\int_0^T|f(r,0)|\,dr\big|\FF_\alpha\big)=:U_\alpha.
\end{align}
By \cite[page 417]{DM1}, $U$ is a process of class (D). Hence, by (H4), there exists a chain $\{\tau_k\}$ on $[0,T]$ such that
\begin{equation}
\label{eq.ad3.1}
E\int_0^{\tau_k}|f(r,U_r)|\,dr + E\int_0^{\tau_k}|f(r,-U_r)|\,dr\le k,\quad k\ge 0.
\end{equation}
By (\ref{eq.ad3.2}), (H2) and Lemma \ref{lm.ad3.1},
\begin{align*}
&E\int_{\tau\wedge \delta_n}^\tau|f(r,\EE^f_{r,\tau}(X_\tau))|\,dr\\
&\qquad \le E\int_{\tau\wedge \delta_n\wedge\tau_k}^{\tau\wedge\tau_k}|f(r,\EE^f_{r,\tau}(X_\tau))|\,dr
+E\mathbf{1}_{\{\tau_k<\tau\}}\int_{\tau_k }^{\tau}|f(r,\EE^f_{r,\tau}(X_\tau))|\,dr\\
&\qquad\le E\int_{\delta_n\wedge\tau_k}^{\tau_k}|f(r,U_r)|\,dr
+ E\int_{\delta_n\wedge\tau_k}^{\tau_k}|f(r,-U_r)|\,dr\\
&\qquad\quad+E\mathbf{1}_{\{\tau_k<\tau\}}
E\big(\int_{\tau_k }^{\tau} |f(r,\EE^f_{r,\tau}(X_\tau))|\,dr\big|\FF_{\tau_k}\big)\\
&\qquad\le E\int_{\delta_n\wedge\tau_k}^{\tau_k}|f(r,U_r)|\,dr + E\int_{\delta_n\wedge\tau_k}^{\tau_k}|f(r,-U_r)|\,dr\\
&\qquad\quad
+E\mathbf{1}_{\{\tau_k<\tau\}}\big(|X_\tau|-|\EE^f_{\tau_k,\tau}(X_\tau)|
+\int_{\tau_k}^T|f(r,0)|\,dr\big)\\
&\qquad
\le E\int_{ \delta_n\wedge\tau_k}^{\tau_k}|f(r,U_r)|\,dr + E\int_{ \delta_n\wedge\tau_k}^{\tau_k}|f(r,-U_r)|\,dr+2E\mathbf{1}_{\tau_k<T}U_{\tau_k}.
\end{align*}
Since $U$ is of class (D) and $\{\tau_k\}$ is a chain on $[0,T]$, there exists $k_0\ge 0$ such that $2E\mathbf{1}_{\{\tau_k<T\}}U_{\tau_k}\le \varepsilon,\, k\ge k_0$.
By (\ref{eq.ad3.1}) and the Lebesgue dominated convergence theorem,
\[
E\int_{ \delta_n\wedge\tau_{k_0}}^{\tau_{k_0}}|f(r,U_r)|\,dr + E\int_{ \delta_n\wedge\tau_{k_0}}^{\tau_{k_0}}|f(r,-U_r)|\,dr\rightarrow 0
\]
as $n\rightarrow\infty$.  By what has been proved,
\[
\limsup_{n\rightarrow \infty}\sup_{\tau\in\TT}E\int_{\tau\wedge \delta_n}^\tau|f(r,\EE^f_{r,\tau}(X_\tau))|\,dr\le \varepsilon.
\]
Since $\varepsilon>0$ was arbitrary, this proves the lemma.
\end{proof}

\nsubsection{Definition of a solution and a comparison result}
\label{sec4}

\begin{definition}\label{pt1}
We say that a triple $(Y,M,K)$ of $\mathbb F$-adapted processes is a solution of the reflected backward stochastic differential equation on the interval $[0,T]$ with right-hand side $f+dV$, terminal value $\xi$ and lower barrier $L$ (\underline{R}BSDE$^T(\xi,f+dV,L)$ for short) if
\begin{enumerate}
\item[(a)] $Y$ is regulated and $M\in\mathcal{M}_{loc}$,
\item[(b)] $K\in\mathcal{V}^+_p$, $L_t\le Y_t$, $t\in[0,T_a]$, $a\ge 0$, and
\[
\int^{T_a}_0(Y_{r-}-\overrightarrow{L}_r)\,dK^*_r+\sum_{0\le r<T_a}(Y_r-L_r)\Delta^+K_r=0,\quad a\ge 0,
\]
\item[(c)] $\int^{T_a}_0|f(r,Y_r)|\,dr<+\infty$ for every $a\ge 0$,
\item[(d)] for every $a\ge 0$,
\begin{align*}
Y_t=Y_{T_a}+\int^{T_a}_t f(r,Y_r)\,dr+\int^{T_a}_t\,dV_r+\int^{T_a}_t\,dK_r-\int^{T_a}_t\,dM_r,\quad t\in[0,T_a],
\end{align*}
\item[(e)] $\lim_{a\rightarrow\infty}Y_{T_a}=\xi$ a.s.
\end{enumerate}
\end{definition}

\begin{definition}
We say that a triple $(Y,M,R)$ of $\mathbb F$-adapted  processes is a solution of the reflected backward stochastic differential equation on the interval $[0,T]$ with right-hand side $f+dV$, terminal value $\xi$, lower barrier $L$ and upper barrier $U$ (RBSDE$^T(\xi,f+dV,L,U)$ for short) if
\begin{enumerate}
\item[(a)] $Y$ is regulated and $M\in\mathcal{M}_{loc}$,
\item[(b)] $R\in\mathcal{V}_p$, $L_t\le Y_t\le U_t$, $t\in[0,T_a]$, $a\ge 0$, and
\begin{align*}
&\int^{T_a}_0(Y_{r-}-\overrightarrow{L}_r)\,dR^{*,+}_r+\sum_{0\le r<T_a}(Y_r-L_r)\Delta^+R^+_r\\
&\qquad+\int^{T_a}_0(\underrightarrow{U}_r-Y_{r-})\,dR^{*,-}_r+\sum_{0\le r<T_a}(U_r-Y_r)\Delta^+R^-_r=0,\quad a\ge 0,
\end{align*}
where $R=R^+-R^-$ is the Jordan decomposition of $R$,
\item[(c)] $\int^{T_a}_0|f(r,Y_r)|\,dr<+\infty$ for every $a\ge 0$,
\item[(d)] for every $a\ge 0$,
\begin{align*}
Y_t=Y_{T_a}+\int^{T_a}_t f(r,Y_r)\,dr+\int^{T_a}_t\,dV_r+\int^{T_a}_t\,dR_r-\int^{T_a}_t\,dM_r,\quad t\in[0,T_a],
\end{align*}
\item[(e)] $\lim_{a\rightarrow\infty}Y_{T_a}=\xi$ a.s.
\end{enumerate}
\end{definition}

\begin{proposition}
\label{sob4}
Let $(Y^i,M^i,R^i)$ be a solution of \textnormal{RBSDE}$^T(\xi^i,f^i+dV^i,L^i,U^i)$, $i=1,2$. Assume that $f^1$ satisfies \textnormal{(H2)} and $\xi^1\le\xi^2$, $f^1(\cdot,Y^2)\le f^2(\cdot,Y^2)$ $dt\otimes dP$-a.s., $dV^1\le dV^2$, $L^1\le L^2$, $U^1\le U^2$. If $(Y^1-Y^2)^+$ is of class \textnormal{(D)}, then $Y^1\le Y^2$.
\end{proposition}
\begin{proof}
By (H2) and the fact that $f^1(\cdot,Y^2)\le f^2(\cdot,Y^2)$ $dt\otimes dP$-a.s. we have
\begin{equation}\label{pon1}
\begin{split}
\mathbf{1}_{\{Y^1_r>Y^2_r\}}(f^1(r,Y^1_r)-f^2(r,Y^2_r))\le \mathbf{1}_{\{Y^1_r>Y^2_r\}}(f^1(r,Y^1_r)-f^1(r,Y^2_r))\le 0.
\end{split}
\end{equation}\label{pon2}
By the minimality condition for $R^1,R^2$ and the assumption that $L^1\le L^2$ and $U^1\le U^2$,
\begin{equation}
\mathbf{1}_{\{Y^1_{r-}>Y^2_{r-}\}}d(R^1_r-R^2_r)^*
\le\mathbf{1}_{\{Y^1_{r-}>Y^2_{r-}\}}\,dR^{1,*,+}
+\mathbf{1}_{\{Y^1_{r-}>Y^2_{r-}\}}\,dR^{2,-,*}=0
\end{equation}
and
\begin{equation}\label{pon3}
\mathbf{1}_{\{Y^1_r>Y^2_r\}}\Delta^+(R^1_r-R^2_r)\le\mathbf{1}_{\{Y^1_r>Y^2_r\}}\Delta^+R_r^{1,+}+\mathbf{1}_{\{Y^1_r>Y^2_r\}}\Delta^+R_r^{2,-}=0.
\end{equation}
By \cite[Corollary A.5]{KRzS}, for all $a\ge0$ and stopping times $\sigma,\tau\in\mathcal{T}_{T_a}$ such that $\sigma\le\tau$ we have
\[
\begin{split}
(Y^1_{\sigma}-Y^2_{\sigma})^+&\le (Y^1_{\tau}-Y^2_{\tau})^++\int^{\tau}_{\sigma}\mathbf{1}_{\{Y^1_r>Y^2_r\}}
(f^1(r,Y^1_r)-f^2(r,Y^2_r))\,dr\\
&\quad+\int^{\tau}_{\sigma}\mathbf{1}_{\{Y^1_{r-}>Y^2_{r-}\}}\,d(V^1_r-V^2_r)^*
+\sum_{\sigma\le r<\tau}\int^{\tau}_{\sigma}\mathbf{1}_{\{Y^1_r>Y^2_r\}}\Delta^+(V^1_r-V^2_r)\\
&\quad+\int^{\tau}_{\sigma}\mathbf{1}_{\{Y^1_{r-}>Y^2_{r-}\}}\,d(R^1_r-R^2_r)^*
+\sum_{\sigma\le r<\tau}\int^{\tau}_{\sigma}\mathbf{1}_{\{Y^1_r>Y^2_r\}}\Delta^+(R^1_r-R^2_r)\\
&\quad-\int^{\tau}_{\sigma}\mathbf{1}_{\{Y^1_{r-}>Y^2_{r-}\}}\,d(M^1_r-M^2_r).
\end{split}
\]
By the above inequality, (\ref{pon1})--(\ref{pon3}) and the assumption that $dV_1\le dV_2$, we have
\begin{equation}\label{pon4}
(Y^1_{\sigma}-Y^2_{\sigma})^+\le(Y^1_{\tau}-Y^2_{\tau})^+-\int^{\tau}_{\sigma}\mathbf{1}_{\{Y^1_{r-}>Y^2_{r-}\}}\,d(M^1_r-M^2_r)
\end{equation}
Let $\{\tau_k\}$ be a localizing sequence, on $[0,T_a]$, for the local martingale $M^1-M^2$. By (\ref{pon4}) with $\tau$ replaced by $\tau_k\ge\sigma$, we get
\begin{equation*}
(Y^1_{\sigma}-Y^2_{\sigma})^+\le(Y^1_{\tau_k}-Y^2_{\tau_k})^+-\int^{\tau_k}_{\sigma}\mathbf{1}_{\{Y^1_{r-}>Y^2_{r-}\}}\,d(M^1_r-M^2_r),\, k\in\mathbb{N}.
\end{equation*}
Taking the expectation and then letting $k\rightarrow\infty$ we obtain $E(Y^1_{\sigma}-Y^2_{\sigma})^+\le(Y^1_{T_a}-Y^2_{T_a})^+$ for $a\ge 0$. Letting $a\rightarrow\infty$ yields $E(Y^1_{\sigma}-Y^2_{\sigma})^+=E(\xi^1-\xi^2)^+=0$. Therefore, by the Section Theorem (see, e.g., \cite[Chapter IV, Theorem 86]{dm}), $(Y^1_t-Y^2_t)^+=0$, $t\in[0,T]$. Hence $Y^1\le Y^2$.
\end{proof}

\nsubsection{Existence of a solution with one reflecting barrier}
\label{sec5}

We will need the following additional assumption.
\begin{enumerate}
\item[(H5)] There exists a process $X$ such that $L\le X$, $E\int^T_0 f^-(r,X_r)\,dr<\infty$ and $X$ is a difference of two supermartingales of class (D) on $[0,T]$.
\end{enumerate}

\subsubsection{Snell envelope of an optional process of class (D)}

In this section, we  recall some properties of Snell envelope of an optional process of class (D).
In the whole section, we assume that $L$ is an optional process of class (D) defined on $[0,T]$.
Given $\alpha\in\TT$, we set
\[
Y(\alpha)=\esssup_{\tau\ge\alpha}E(L_\tau|\FF_\alpha).
\]
Let $M_t=E(L_T|\FF_t)$ and $\hat L= L-M$.  We set
\[
\hat Y(\alpha):=\esssup_{\tau\ge\alpha}E(\hat L_\tau|\FF_\alpha).
\]
Observe that   $\hat Y=Y-M$. Moreover, $\hat L_T=0$.  By \cite[page 417]{DM1}, there exists a positive supermartingale $\hat Y$ of class (D) on $[0,T]$
such that for every $\alpha\in \TT$,
\[
\hat Y(\alpha)=\hat Y_\alpha.
\]
Therefore $Y:= \hat Y+M$ is also a supermartingale of class (D) on $[0,T]$ and for every $\alpha\in\TT$,
\[
Y(\alpha)=Y_\alpha.
\]
By \cite[page 417]{DM1}, $Y$ is the smallest supermartingale majorizing $L$.

\begin{lemma}
\label{lm.ad5.1}
For all $\alpha, \sigma\in \TT$ such that $\alpha\le \sigma$,
\[
Y_\alpha=\esssup_{\tau\in \TT_{\alpha,\sigma}}E(L_{\tau}\mathbf{1}_{\tau<\sigma}
+Y_\sigma\mathbf{1}_{\tau=\sigma}|\FF_\alpha).
\]
\end{lemma}
\begin{proof}
By the considerations preceding Lemma \ref{lm.ad5.1} we may and will assume that $Y$ is positive.  By the fact that $Y$ is a supermartingale and $L\le Y$,
\begin{align}
\label{eq.ad5.1}
\nonumber Y_\alpha&=\esssup_{\tau\ge \alpha}E(L_\tau\mathbf{1}_{\tau<\sigma}+L_\tau \mathbf{1}_{\tau\ge \sigma}|\FF_\alpha)
=\esssup_{\tau\ge \alpha}E(L_\tau\mathbf{1}_{\tau<\sigma}+ \mathbf{1}_{\tau\ge \sigma}E(L_\tau|\FF_\sigma)|\FF_\alpha)
\nonumber\\&\le  \esssup_{\tau\ge \alpha}E(L_\tau\mathbf{1}_{\tau<\sigma}+ \mathbf{1}_{\tau\ge \sigma}E(Y_\tau|\FF_\sigma)|\FF_\alpha)
\le \esssup_{\tau\ge \alpha}E(L_\tau\mathbf{1}_{\tau<\sigma}+ \mathbf{1}_{\tau\ge \sigma}Y_\sigma|\FF_\alpha) \nonumber\\&
= \esssup_{\tau\in\TT_{\alpha,\sigma}}E(L_\tau\mathbf{1}_{\tau<\sigma}+ \mathbf{1}_{\tau= \sigma}Y_\sigma|\FF_\alpha).
\end{align}
Since   $Y$ is positive, for every $\tau\in\TT_{\alpha,\sigma}$ we have
\begin{align*}
E(L_{\tau}\mathbf{1}_{\tau<\sigma}+Y_\sigma\mathbf{1}_{\tau=\sigma}|\FF_\alpha)&\le E(Y_{\tau}\mathbf{1}_{\tau<\sigma}+Y_\tau\mathbf{1}_{\tau=\sigma}|\FF_\alpha)\\&
=E(Y_\tau\mathbf{1}_{\tau\le \sigma}|\FF_\alpha)\le E(Y_\tau|\FF_\alpha)\le Y_\alpha.
\end{align*}
Hence
\[
\esssup_{\tau\in\TT_{\alpha,\sigma}}
E(L_{\tau}\mathbf{1}_{\tau<\sigma}+Y_\sigma\mathbf{1}_{\tau=\sigma}|\FF_\alpha)\le Y_\alpha,
\]
which when combined with (\ref{eq.ad5.1}) proves the lemma.
\end{proof}

\subsubsection{Existence results for linear equations }

\begin{theorem}\label{th1}
Assume that $f$ is independent of $y$. Assume also that $\xi\in L^1$, $V\in\mathcal{V}^1$ and that $E\int^T_0|f(r)|\,dr<\infty$. Then there exists a unique solution $(Y,M,K)$ of \textnormal{\underline{R}BSDE}$^T(\xi,f+dV,L)$, such that $Y$ is of class \textnormal{(D)}. Moreover,
for every $\alpha\in\TT$,
\end{theorem}
\begin{equation*}
Y_\alpha=\esssup_{\tau\in\mathcal{T}_{\alpha}}E\Big(\int^{\tau}_\alpha f(r)\,dr+\int^{\tau}_\alpha\,dV_r+L_{\tau}\mathbf{1}_{\{\tau<T\}}+\xi\mathbf{1}_{\{\tau=T\}}|\mathcal{F}_\alpha\Big).
\end{equation*}
\begin{proof}
For a stopping time $\alpha\in\mathcal{T}$, we put
\begin{equation}\label{th1.1}
Y_{\alpha}=\esssup_{\tau\in\mathcal{T}_{\alpha}}E\Big(\int^{\tau}_{\alpha} f(r)\,dr+\int^{\tau}_{\alpha}\,dV_r+L_{\tau}\mathbf{1}_{\{\tau<T\}}
+\xi\mathbf{1}_{\{\tau=T\}}|\mathcal{F}_{\alpha}\Big).
\end{equation}
Observe that
\begin{align}
\label{sob5}
S_{\alpha}:=Y_{\alpha}+\int^{\alpha}_0f(r)\,dr+\int^{\alpha}_0\,dV_r
&=\esssup_{\tau\in\mathcal{T}_{\alpha}}E\Big(\int^{\tau}_0 f(r)\,dr
+\int^{\tau}_0\,dV_r\nonumber\\
&\quad+L_{\tau}\mathbf{1}_{\{\tau<T\}}
+\xi\mathbf{1}_{\{\tau=T\}}|\mathcal{F}_{\alpha}\Big)
\end{align}
is the Snell envelope of the process $\hat{L}_t:=\int^t_0 f(r)\,dr+\int^t_0\,dV_r+L_t\mathbf{1}_{\{t<T\}}+\xi\mathbf{1}_{\{t\wedge T=T\}}$. By \cite[page 417]{DM1}, $S$ is a supermartingale of class (D).  By the Mertens decomposition theorem (see \cite{Mertens}),  there exist $K\in\mathcal{V}^+_p$ and $M\in\mathcal{M}_{loc}$ such that
\begin{equation}\label{sob6}
S_t=S_0-\int^t_0\,dK_r+\int^t_0\,dM_r,\quad t\le T.
\end{equation}
Since $S$ is of class (D) and $Y=S-\int^{\cdot}_0 f(r)\,dr-\int^{\cdot}_0\,dV_r$, $Y$ is a regulated process of class (D). Moreover, by (\ref{sob5}), $Y\ge L$, and from (\ref{sob6}) it follows that for every $a\ge 0$,
\begin{align*}
Y_t=Y_{T_a}+\int^{T_a}_t f(r)\,dr+\int^{T_a}_t\,dV_r+\int^{T_a}_t\,dK_r-\int^{T_a}_t\,dM_r,\quad t\in[0,T_a].
\end{align*}
We will show that $(Y,M,K)$ is a solution of \textnormal{\underline{R}BSDE}$^T(\xi,f+dV,L)$. To check this it remains to show the minimality condition (b) and condition (d) of Definition \ref{pt1} are satisfied. Applying Lemma \ref{lm.ad5.1} and  \cite[Lemma 3.2, Lemma 3.3]{GIOQ} on each interval $[0,T_a]$ we show that $K\in\mathcal{V}^+_p$ satisfies (b).  By \cite[Lemma 3.8]{Kl}, we have
\begin{equation}\label{sob1}
\lim_{a\rightarrow\infty}Y_{T_a}\le\xi.
\end{equation}
On the other hand, for $\tau=T$,
\[
Y_\alpha\ge E\Big(\int^T_\alpha f(r)\,dr+\int^T_\alpha\,dV_r+\xi|\mathcal{F}_\alpha\Big).
\]
From \cite[Remark 2.2]{Kl} we know that the right-hand side of the above inequality is the first component of the solution to BSDE$^T(\xi,f+dV)$. Hence
\begin{equation}\label{sob3}
\lim_{a\rightarrow\infty}Y_{T_a}\ge\xi.
\end{equation}
By (\ref{sob1}) and (\ref{sob3}), $\lim_{a\rightarrow\infty}Y_{T_a}=\xi$, which proves that
$(Y,M,K)$ is a solution of \textnormal{\underline{R}BSDE}$^T(\xi,f+dV,L)$. Uniqueness follows from Proposition \ref{sob4}.
\end{proof}

\subsubsection{A priori estimates and Snell envelope representation for solutions to nonlinear RBSDEs}

For an $\mathbb F$-adapted regulated process $Y$, we set
\[
f_Y(t)=f(t,Y_t),\,t\in[0,T].
\]
\begin{proposition}\label{prop2}
Let $(Y,M,K)$ be a solution of \textnormal{\underline{R}BSDE}$^T(\xi,f+dV,L)$ such that $Y$ is of class \mbox{\rm(D)} and $E\int^T_0|f(r,Y_r)|\,dr<\infty$. Assume also that $\xi\in L^1$ and $V\in\mathcal{V}^1$. Then for every $\alpha\in\mathcal{T}$,
\[
Y_{\alpha}=\esssup_{\tau\in\mathcal{T}_ {\alpha}}E\Big(\int^{\tau}_{\alpha} f(r,Y_r)\,dr+\int^{\tau}_{\alpha}\,dV_r+L_{\tau}\mathbf{1}_{\{\tau<T\}}
+\xi\mathbf{1}_{\{\tau=T\}}|\mathcal{F}_{\alpha}\Big).
\]
\end{proposition}
\begin{proof}
It is clear that $(Y,M,K)$ is a solution of \textnormal{\underline{R}BSDE}$^T(\xi,f_Y+dV,L)$. Set
\begin{equation}\label{prop2.1}
\bar{Y}_{\alpha}=\esssup_{\tau\in\mathcal{T}_{\alpha}}E\Big(\int^{\tau}_{\alpha} f_Y(r)\,dr+\int^{\tau}_{\alpha}\,dV_r+L_{\tau}\mathbf{1}_{\{\tau<T\}}+\xi\mathbf{1}_{\{\tau=T\}}|\mathcal{F}_{\alpha}\Big).
\end{equation}
By Theorem \ref{th1} and (\ref{th1.1}), $\bar{Y}$ is the first component of the solution to \textnormal{\underline{R}BSDE}$^T(\xi,f_Y+dV,L)$ such that $\bar{Y}$ is of class (D). By uniqueness, $\bar{Y}=Y$, so(\ref{prop2.1}) yields the required result.
\end{proof}

\begin{proposition}\label{prop3}
Assume that \textnormal{(H1)--(H5)} are satisfied. Let $(Y,M,K)$ be a solution of \textnormal{\underline{R}BSDE}$^T(\xi,f+dV,L)$. Then
\begin{equation*}
\begin{split}
E\int^T_0|f(r,Y_r)|\,dr+EK_T&\le C\Big(\|Y\|_1+\|X\|_1+E\int^T_0|f(r,0)|\,dr\\
&\quad+E\int^T_0\,d|V|_r
+E\int^T_0f^-(r,X_r)\,dr+E\int^T_0\,d|C|_r\Big),
\end{split}
\end{equation*}
where $X_t=X_0+C_t+H_t,\, t\ge 0$, is the Doob-Meyer decomposition of the process $X$ appearing in \textnormal{(H5)}.
Moreover, $M$ is a uniformly integrable martingale and
\[
M_t=E\Big(\int^T_0 f(r,Y_r)\,dr+\int^T_0\,dV_r+\int^T_0\,dK_r|\mathcal{F}_t\Big)-Y_0,\quad t\in[0,T].
\]
\end{proposition}
\begin{proof}
Let $(\underline{Y},\underline{M})$ be a solution of BSDE$^T(\xi,f+dV)$ such that $\underline{Y}$ is of class (D). The existence of the solution follows from \cite[Proposition 2.7]{Kl}. By Proposition \ref{sob4}, $Y\ge\underline{Y}$. By (H5), there exists a process $X$ such that $X$ is of class (D), $X\ge L$ and $E\int^T_0 f^-(r,X_r)\,dr<\infty$. There exist processes $H\in\mathcal{M}_{loc}$ and $C\in\mathcal{V}^1_p$ such that for every $a\ge0$,
\[
X_t=X_{T_a}+\int^{T_a}_t\,dC_r-\int^{T_a}_t\,dH_r,\quad t\in[0,T_a].
\]
This equation can be rewritten as
\begin{equation*}
\begin{split}
X_t&=X_{T_a}+\int^{T_a}_t f(r,X_r)\,dr+\int^{T_a}_t\,dV_r+\int^{T_a}_t\,dC'_r-\int^{T_a}_t\,dH_r, \quad t\in[0,T_a],
\end{split}
\end{equation*}
with $C'_t=-\int^t_0 f(r,X_r)\,dr-V_t+C_t$. Let $(\bar{X},\bar{H})$ be a solution of BSDE($\xi$,$f+dV^++dC^{',+}$) such that $\bar{X}$ is of class (D). The existence of the solution follows from \cite[Proposition 2.7]{Kl}. By Proposition \ref{sob4}, $\bar{X}\ge X$, so $\bar{X}\ge L$. Note that the triple $(\bar{X},\bar{H},0)$ is a solution of {\underline{R}BSDE}$^T(\xi$,$f+dV^++dC^{',+}$,$\bar{X})$. Hence, by Proposition \ref{sob4}, $\bar{X}\ge Y$. We have
\begin{equation}\label{prop3.0}
\underline{Y}\le Y\le\bar{X},
\end{equation}
so $Y$ is of class (D). Furthermore, by (H2) and (\ref{prop3.0}),
\begin{equation}\label{prop3.1}
f(r,\underline{Y})\ge f(r,Y_r)\ge f(r,\bar{X}).
\end{equation}
By (H5), (\ref{prop3.1}) and \cite[Theorem 2.8]{Kl},
\begin{align}
\label{prop3.1.1}
E\int^T_0|f(r,Y_r)|\,dr&\le C\Big(E|\xi|+E|X_T|+E\int^T_0|f(r,0)|\,dr \nonumber\\
&\quad+E\int^T_0f^-(r,X_r)\,dr+E\int^T_0\,d|V|_r+E\int^T_0\,d|C|_r\Big)
\end{align}
for some $C_1>0$. We will show that $EK_T<\infty$. For every $a\ge0$ we have
\begin{equation}\label{prop3.2}
Y_{T_a}=Y_0-\int^{T_a}_0 f(r,Y_r)\,dr-\int^{T_a}_0\,dV_r-\int^{T_a}_0\,dK_r+\int^{T_a}_0\,dM_r.
\end{equation}
Let $\{\tau_k\}$ be a localizing sequence on $[0,T_a]$ for the local martingale $M$. By (\ref{prop3.2}) with $T_a$ replaced by $\tau_k$,
\begin{equation*}
Y_{\tau_k}=Y_0-\int^{\tau_k}_0 f(r,Y_r)\,dr-\int^{\tau_k}_0\,dV_r-\int^{\tau_k}_0\,dK_r+\int^{\tau_k}_0\,dM_r.
\end{equation*}
Taking the expectation and letting $k\rightarrow\infty$ and then $a\rightarrow\infty$,  we obtain
\begin{equation*}\label{prop3.3}
EY_T=EY_0-E\int^T_0 f(r,Y_r)\,dr-E\int^T_0\,dV_r-E\int^T_0\,dK_r
\end{equation*}
and
\begin{equation*}
EK_T\le \Big(2\|Y\|_1+E\int^T_0 f(r,Y_r)\,dr+E\int^T_0\,d|V|_r\Big).
\end{equation*}
By (H2),
\begin{equation*}
EK_T\le \Big(2\|Y\|_1+E\int^T_0 |f(r,0)|\,dr+E\int^T_0\,d|V|_r\Big).
\end{equation*}
By the above inequality and (\ref{prop3.1.1}),
\begin{equation*}
\begin{split}
E\int^T_0|f(r,Y_r)|\,dr+EK_T&\le C\Big(\|Y\|_1+E|\xi|+E|X_T|+E\int^T_0|f(r,0)|\,dr\\
&\quad+E\int^T_0f^-(r,X_r)\,dr+E\int^T_0\,d|V|_r+E\int^T_0\,d|C|_r\Big)
\end{split}
\end{equation*}
for some $C>0$. By \cite[Remark 2.2]{Kl},
\[
M_t=E\Big(\int^T_0 f(r,Y_r)\,dr+\int^T_0\,dV_r+\int^T_0\,dK_r|\mathcal{F}_t\Big)-Y_0.
\]
so $M$ is a uniformly integrable martingale. This completes the proof.
\end{proof}

\subsubsection{Existence results for nonlinear RBSDEs}

\begin{proposition}\label{th2}
Assume that \textnormal{(H1), (H2), (H4)} are satisfied and there is a measurable $\lambda:[0,\infty)\rightarrow\mathbb{R}^+$ such that $\int^\infty_0\lambda(r)\,dr<\infty$ and
\[
|f(t,y_1)-f(t,y_2)|\le\lambda|y_1-y_2|,\quad t\in[0,T],\, y_1,y_2\in\BR.
\]
Then there exists a unique solution $(Y,M,K)$ of \textnormal{\underline{R}BSDE}$^T(\xi,f+dV,L)$ such that $Y$ is of class \textnormal{(D)}.
\end{proposition}
\begin{proof}
Let $(Y^0,M^0)$ be a solution of BSDE$^T(\xi$,$f+dV)$ such that $Y^0$ is of class (D). The existence of the solution follows from \cite[Proposition 2.7]{Kl}. Next, for each $n\ge 1$, we define $(Y^n,M^n,K^n)$ to be a solution of {\underline{R}BSDE}$^T(\xi,f_n+dV,L)$ with
\[
f_n(r)=f(r,Y^{n-1}_r),
\]
such that $Y^n$ is of class (D). The existence of such a solution follows from Theorem \ref{th1}. By \cite[Theorem 2.8]{Kl},
\begin{equation}\label{th2.0.0}
E\int^T_0 |f(r,Y^0)|\,dr<\infty.
\end{equation}
Since $Y^n$ is of class (D), it follows from (H1) and the assumption that $f$ is Lipschitz continuous that
\begin{equation}\label{th2.0}
E\int^T_0|f(r,Y^n_r)|\,dr\le\int^T_0\lambda(r)\,dr\|Y^n\|_1+E\int^T_0|f(r,0)|\,dr<\infty.
\end{equation}
By Proposition \ref{sob4},
\begin{equation}\label{th2.1}
Y^0\le Y^2\le Y^4\le Y^6\le\dots\le Y^5\le Y^3\le Y^1.
\end{equation}
Consider the  monotone subsequences $\{Y^{2k}\}_{k\in\mathbb{N}}$ and $\{Y^{2k+1}\}_{k\in\mathbb{N}}$ of $\{Y^n\}_{n\in\mathbb{N}}$, and set $\bar{Y}=\lim_{k\rightarrow\infty}Y^{2k}$, $\underline{Y}=\lim_{k\rightarrow\infty}Y^{2k+1}$. We shall show that $\bar{Y}$ is a solution of {\underline{R}BSDE}$^T(\xi,\bar{f}_n+dV,L)$ with $\bar{f}(r)=f(r,\underline{Y}_r)$ and $\underline{Y}$ is a solution of {\underline{R}BSDE}$^T(\xi,\underline{f}+dV,L)$ with $\underline{f}(r)=f(r,\bar{Y}_r)$.
By (\ref{th2.1}), $\bar{Y}$ and $\underline{Y}$ are of class (D) and $\lim_{a\rightarrow\infty}\bar{Y}_{T_a}=\lim_{a\rightarrow\infty}\underline{Y}_{T_a}=\xi$. By (H2) and (\ref{th2.1}), for each $n\ge 1$,
\begin{equation}\label{th2.2}
f(r,Y^0_r)\ge f(r,Y^n_r)\ge f(r,Y^1_r),
\end{equation}
and by (H3),
\begin{equation}\label{th2.3}
f(r,Y^{2k-1}_r)\rightarrow f(r,\underline{Y}_r).
\end{equation}
Using  (\ref{th2.0.0}), (\ref{th2.0}), (\ref{th2.2}), (\ref{th2.3}) and the Lebesgue dominated convergence theorem shows that
\begin{equation}\label{th2.4}
E\int^{T}_0|f(r,Y^{2k-1}_r)-f(r,\underline{Y}_r)|\,dr\rightarrow 0
\end{equation}
as $k\rightarrow\infty$. By Proposition \ref{prop2}, for $\sigma\in\mathcal{T}_T$,
\begin{equation}\label{th2.5}
Y^{2k}_{\sigma}=\esssup_{\tau\in\mathcal{T}_{\sigma}}E\Big(\int^{\tau}_{\sigma} f(r,Y^{2k-1}_r)\,dr+\int^{\tau}_{\sigma}\,dV_r+L_{\tau}\mathbf{1}_{\{\tau<T\}}
+\xi\mathbf{1}_{\{\tau=T\}}|\mathcal{F}_{\sigma}\Big).
\end{equation}
By (\ref{th2.4}), (\ref{th2.5}) and \cite[Lemma 3.19]{KRzS},
\begin{equation*}
\bar{Y}_{\sigma}=\esssup_{\tau\in\mathcal{T}_{\sigma}}E\Big(\int^{\tau}_{\sigma} f(r,\underline{Y}_r)\,dr+\int^{\tau}_{\sigma}\,dV_r+L_{\tau}\mathbf{1}_{\{\tau<T\}}+\xi\mathbf{1}_{\{\tau=T\}}|\mathcal{F}_{\sigma}\Big).
\end{equation*}
By similar arguments we have, for every $\sigma\in\mathcal{T}_T$,
\begin{equation*}
\underline{Y}_{\sigma}=\esssup_{\tau\in\mathcal{T}_{\sigma}}E\Big(\int^{\tau}_{\sigma} f(r,\bar{Y}_r)\,dr+\int^{\tau}_{\sigma}\,dV_r+L_{\tau}\mathbf{1}_{\{\tau<T\}}+\xi\mathbf{1}_{\{\tau=T\}}|\mathcal{F}_{\sigma}\Big).
\end{equation*}
By \cite[Lemma 3.1, Lemma 3.2]{GIOQ}, the processes $\bar{Y},\underline{Y}$ are regulated and there exist $\bar{K},\underline{K}\in\mathcal{V}^+_p$ and $\bar{M},\underline{M}\in\mathcal{M}_{loc}$ such that for all $a\ge0$ and $t\in[0,T_a]$,
\[
\bar{Y}_t=\bar{Y}_{T_a}+\int^{T_a}_t f(r,\underline{Y}_r)\,dr+\int^{T_a}_t\,dV_r+\int^{T_a}_t\,d\bar{K}_r-\int^{T_a}_t\,d\bar{M}_r,
\]
\[
\underline{Y}_t=\underline{Y}_{T_a}+\int^{T_a}_t f(r,\bar{Y}_r)\,dr+\int^{T_a}_t\,dV_r+\int^{T_a}_t\,d\underline{K}_r
-\int^{T_a}_t\,d\underline{M}_r.
\]
Moreover, by \cite[Lemma 3.3]{GIOQ} and \cite[Proposition 2.34, p.~131]{EK}, we have
\[
\int^{T_a}_0(\bar{Y}_{r-}-\overrightarrow{L}_r)\,d\bar{K}^*_r+\sum_{0\le r<T_a}(\bar{Y}_r-L_r)\Delta^+\bar{K}_r=0,
\]
%and for $\underline{K}$
\[
\int^{T_a}_0(\underline{Y}_{r-}-\overrightarrow{L}_r)\,d\underline{K}^*_r+\sum_{0\le r<T_a}(\underline{Y}_r-L_r)\Delta^+\underline{K}_r=0
\]
for all $a\ge0$ and $t\in[0,T_a]$. Therefore ($\bar{Y},\bar{M},\bar{K}$) is a solution of {\underline{R}BSDE}$^T(\xi,\bar{f}+dV,L)$, and ($\underline{Y},\underline{M},\underline{K}$) is a solution of {\underline{R}BSDE}$^T(\xi,\underline{f}+dV,L)$.
We will show that $\bar{Y}=\underline{Y}$. By \cite[Corollary A.5]{KRzS}, for all $a\ge0$ and stopping times $\sigma,\tau\in\mathcal{T}_{T_a}$ such that $\sigma\le\tau$ we have
\begin{align}
\label{th2.6}
|\bar{Y}_{\sigma}-\underline{Y}_{\sigma}|&\le |\bar{Y}_{\tau}-\underline{Y}_{\tau}|
+\int^{\tau}_{\sigma}\sgn_{\{\bar{Y}_r>\underline{Y}_r\}}
(f(r,\bar{Y}_r)-f(r,\underline{Y}_r))\,dr\nonumber\\
&\quad+\int^{\tau}_{\sigma}\sgn_{\{\bar{Y}_{r-}>\underline{Y}_{r-}\}}\,d(\bar{K}_r
-\underline{K}_r)^*+\sum_{\sigma\le r<\tau}\sgn_{\{\bar{Y}_r>\underline{Y}_r\}} \Delta^+(\bar{K}_r-\underline{K}_r)\nonumber\\
&\quad-\int^{\tau}_{\sigma}\sgn_{\{\bar{Y}_{r-}>\underline{Y}_{r-}\}}\,d(\bar{M}_r
-\underline{M}_r).
\end{align}
By the minimality condition for $\bar{K}$ and the fact that $L\le \bar{Y}$,
\begin{equation}\label{th2.7}
\sgn_{\{\bar{Y}_r>\underline{Y}_r\}}\,d(\bar{K}_r-\underline{K}_r)^*
\le\mathbf{1}_{\{\bar{Y}_r>L_r\}}\,d\bar{K}^*=0.
\end{equation}
and
\begin{equation}\label{th2.8}
\sgn_{\{\bar{Y}_r>\underline{Y}_r\}}\Delta^+(\bar{K}_r-\underline{K}_r)
\le\mathbf{1}_{\{\bar{Y}_r>L_r\}}\Delta^+\bar{K}_r=0.
\end{equation}
By the assumption on $f$ we have that
\begin{equation}\label{th2.9}
\sgn_{\{\bar{Y}_r>\underline{Y}_r\}}(f(r,\bar{Y}_r)-f(r,\underline{Y}_r))\le\lambda(r)|\bar{Y}_r-\underline{Y}_r|.
\end{equation}
By (\ref{th2.6})--(\ref{th2.9}),
\begin{equation}\label{th2.10}
\begin{split}
|\bar{Y}_{\sigma}-\underline{Y}_{\sigma}|&\le |\bar{Y}_{\tau}-\underline{Y}_{\tau}|+\int^{\tau}_0\lambda(r)|\bar{Y}_r
-\underline{Y}_r|\,dr-\int^{\tau}_{\sigma}
\sgn_{\{\bar{Y}_{r-}>\underline{Y}_{r-}\}}\,d(\bar M_r-\underline{M}_r).
\end{split}
\end{equation}
Let $\{\tau_k\}$ be a localizing sequence on $[0,T_a]$ for the local martingale $\int^{\cdot}_{\sigma}\sgn_{\{\bar{Y}_{r-}>\underline{Y}_{r-}\}}\,d(\bar{M}_r-\underline{M}_r)$. By (\ref{th2.10}) with $\tau$ replaced by $\tau_k\ge\sigma$,
\begin{equation}\label{th2.11}
\begin{split}
|\bar{Y}_{\sigma}-\underline{Y}_{\sigma}|&\le |\bar{Y}_{\tau_k}-\underline{Y}_{\tau_k}|+\int^{\tau_k}_0\lambda(r)|\bar{Y}_r
-\underline{Y}_r|\,dr\\
&\quad-\int^{\tau_k}_{\sigma}
\sgn_{\{\bar{Y}_{r-}>\underline{Y}_{r-}\}}\,d(\bar{M}_r-\underline{M}_r),
\end{split}
\end{equation}
for $k\in\mathbb{N}$. Taking the expectation and then letting $k\rightarrow\infty$ and using  Fubini's theorem we get
\[
E|\bar{Y}_{\sigma}-\underline{Y}_{\sigma}|\le E|\bar{Y}_{T_a}-\underline{Y}_{T_a}|+\int^{T_a}_{\sigma}\lambda(r)E|\bar{Y}_r
-\underline{Y}_r|\,dr,\,a\ge 0.
\]
for all $a\ge0$. Applying now Gronwall's lemma with $\sigma$ replaced by $\sigma\wedge t$ and $t\in[0,T_a]$ we get
\begin{equation}
E|\bar{Y}_{\sigma\wedge t}-\underline{Y}_{\sigma\wedge t}|\le E|\bar{Y}_{T_a}-\underline{Y}_{T_a}|\exp\Big(\int^{T_a}_0\lambda(r)\,dr\Big).
\end{equation}
Taking $t=0$ and letting $a\rightarrow\infty$ yields $E|\bar{Y}_{\sigma}-\underline{Y}_{\sigma}|=0$. Therefore, by the Section Theorem (see, e.g., \cite[Chapter IV, Theorem 86]{dm}), $|\bar{Y}_t-\underline{Y}_t|=0$, $t\in[0,T]$. Consequently,  $\bar{Y}=\underline{Y}$ and the triple $(\bar{Y},\bar{M},\bar{K})$ is a solution of {\underline{R}BSDE}$^T(\xi,f+dV,L)$. Uniqueness of the solution follows from Proposition \ref{sob4}.
\end{proof}

\begin{proposition}\label{th3}
Assume that \textnormal{(H1)--(H4)} are satisfied and there exists a progressively measurable process $g$ such that $E\int^T_0|g(r)|\,dr<\infty$ and $f(t,y)\ge g(t)$, $t\in[0,T]$, $y\in\mathbb{R}$. Then there exists a solution $(Y,M,K)$ of \textnormal{{\underline{R}BSDE}}$^T(\xi,f+dV,L)$ such that $Y$ is of class \textnormal{(D)}.
\end{proposition}
\begin{proof}
For each $n\ge 1$ let
\[
f_n(r,y)=c_n(r)\inf_{x\in\mathbb{R}}\{f(r,x)+n|y-x|\},
\]
where $0\le c_n\le 1$, $c_n(r)\nearrow 1$ as $n\rightarrow\infty$ and $\int^T_0 c_n(r)\,dr<\infty$. Let $(Y^n,M^n,K^n)$ be a solution of {\underline{R}BSDE}$^T(\xi,f_n+dV,L)$ such that $Y^n$ is of class (D). It is easy to check that for each $n\ge 1$ the hypotheses  \textnormal{(H1), (H2) and (H4)} are satisfied and for all $t\in[0,T]$ and $y_1,y_2\in\mathbb{R}$,
\[
|f_n(t,y_1)-f(t,y_2)|\le c_n\,n\,|y_1-y_2|.
\]
Therefore the existence of such  solutions follows from Proposition \ref{th2}. Moreover, since for each $n\ge 1$, $f_n\le f_{n+1}$, we have that $Y^n\le Y^{n+1}$ by Proposition \ref{sob4}. Set $Y=\sup_{n\ge 1}Y^n$. Then $Y$ is of class (D) and $\lim_{a\rightarrow\infty}Y_{T_a}=\xi$. To see this, consider the solution $(X,H,C)$ of \underline{R}BSDE$^T(\xi,0,L)$ such that $X$ is of class (D).  The existence of the solution follows from Proposition \ref{th2}. We know that $X$ is supermartingale majorizing $L$ and since $f(t,y)\ge g(t)$ for $t\in[0,T]$ and $y\in\mathbb{R}$, we have  $E\int^T_0 f^-(r,X_r)\,dr<\infty$. Moreover,  $EC_T<\infty$. By the definition of a solution of \underline{R}BSDE we know that for every $a\ge0$,
\[
X_t=X_{T_a}+\int^{T_a}_t\,dC_r-\int^{T_a}_t\,dH_r,\quad t\in[0,T_a].
\]
This equation can be rewritten in the form
\[
X_t=X_{T_a}+\int^{T_a}_t f(r,X_r)\,dr+\int^{T_a}_t\,dV_r+\int^{T_a}_t\,dC'_r-\int^{T_a}_t\,dH_r, \quad t\in[0,T_a],
\]
with  $C'_t=-\int^t_0 f(r,X_r)\,dr-V_t+C_t$. Let $(\bar{X},\bar{H})$ be a solution of BSDE$^T(\xi$,$f+dV^++dC^{',+}$) such that $\bar{X}$ is of class (D). The existence of the solution follows from \cite[Proposition 2.7]{Kl}. Note that $(X,H)$ is a solution of BSDE$^T(\xi,f+dV+dC')$. By Proposition \ref{sob4}, $\bar{X}\ge X$, so $\bar{X}\ge L$. Moreover, the triple $(\bar{X},\bar{H},0)$ is a solution of {\underline{R}BSDE}$^T(\xi$,$f+dV^++dC^{',+}$,$\bar{X})$.  Therefore, by Proposition \ref{sob4}, $\bar{X}\ge Y^n$, $n\ge 1$. We have
\begin{equation}\label{th3.0}
Y^1\le Y^n\le\bar{X},
\end{equation}
so $Y$ is of class (D) and $\lim_{a\rightarrow\infty}Y_{T_a}=\xi$.
By (H3), $f_n(r,Y^n_r)\rightarrow f(r,Y_r)$ as $n\rightarrow\infty$. Since $f_n\le f_{n+1}$, from  (H2), (\ref{th3.0}) and the assumption on $f$ it also follows that
\begin{equation}\label{th3.1}
g(r)\le f_n(r,Y^n_r)\le f(r,Y^1_r).
\end{equation}
Set
\[
\tau_k=\inf\{t\ge 0;\,\int^t_0|f(r,Y^1_r)|\,dr\ge k\}\wedge T.
\]
Observe that $\{\tau_k\}$ is a chain on $[0,T]$ and the triple $(Y^n,M^n,K^n)$ is a solution of {\underline{R}BSDE}($Y^n_{\tau_k}$,$f_n+dV$,$L$) on $[0,\tau_k]$. Hence, by Proposition \ref{prop2}, for every $\sigma\in\mathcal{T}_{\tau_k}$,
\begin{equation}
\label{th3.3}
Y^n_{\sigma}=\esssup_{\tau\in\mathcal{T}_{\sigma,\tau_k}}E\Big(\int^{\tau}_{\sigma} f_n(r,Y^n_r)\,dr+\int^{\tau}_{\sigma}\,dV_r+L_{\tau}\mathbf{1}_{\{\tau<\tau_k\}}
+Y^n_{\tau_k}\mathbf{1}_{\{\tau=\tau_k\}}|\mathcal{F}_{\sigma}\Big).
\end{equation}
By the definition of $\tau_k$ and (\ref{th3.1}),
\begin{equation}\label{th3.4}
E\int^{\tau_k}_0|f_n(r,Y^n_r)-f(r,Y_r)|\,dr\rightarrow 0
\end{equation}
as $n\rightarrow\infty$. By (\ref{th3.0}), (\ref{th3.3}), (\ref{th3.4}) and \cite[Lemma 3.19]{KRzS},
\[
Y_{\sigma}=\esssup_{\tau\in\mathcal{T}_{\sigma,\tau_k}}E\Big(\int^{\tau}_{\sigma} f(r,Y_r)\,dr+\int^{\tau}_{\sigma}\,dV_r
+L_{\tau}\mathbf{1}_{\{\tau<\tau_k\}}
+Y_{\tau_k}\mathbf{1}_{\{\tau=\tau_k\}}|\mathcal{F}_{\sigma}\Big)
\]
for $\sigma\in\mathcal{T}_{\tau_k}$. By \cite[Lemma 3.1, Lemma 3.2]{GIOQ}, $Y$ is regulated and there exist $K^k\in\mathcal{V}^+_p$ and $M^k\in\mathcal{M}_{loc}$ such that for all $a\ge0$ and $t\in\tau_k\wedge a]$,
\begin{align*}
Y_t=Y_{\tau_k\wedge a}+\int^{\tau_k\wedge a}_t f(r,Y_r)\,dr+\int^{\tau_k\wedge a}_t\,dV_r+\int^{\tau_k\wedge a}_t\,dK^k_r-\int^{\tau_k\wedge a}_t\,dM^k_r.
\end{align*}
Also, by \cite[Lemma 3.3]{GIOQ} and \cite[Proposition 2.34, p.~131]{EK}, foe every $a\ge0$ we have
\[
\int^{\tau_k\wedge a}_0(Y_{r-}-\overrightarrow{L}_r)\,dK^{k,*}_r+\sum_{0\le r<\tau_k\wedge a}(Y_r-L_r)\Delta^+K^k_r=0,\quad t\in[0,\tau_k\wedge a].
\]
Therefore $(Y,M^k,K^k)$ is a solution of {\underline{R}BSDE}($Y_{\tau_k}$,$f+dV$,$L$) on $[0,\tau_k]$. By uniqueness, for every $a\ge0$, $K^k_t=K^{k+1}_t$ and $M^k_t=M^{k+1}_t$ for $t\in[0,\tau_k\wedge a]$. Therefore, since $\{\tau_k\}$ is a chain, we can define processes $K$ and $M$ on each interval $[0,T_a]$ by putting $K_t=K^k_t$, $M_t=M^k_t$ on $[0,\tau_k\wedge a]$, $a\ge 0$. Since $\lim_{a\rightarrow\infty}Y_{T_a}=\xi$, we can see that $(Y,M,K)$ is a solution of {\underline{R}BSDE}$^T(\xi,f+dV,L)$. Uniqueness follows from Proposition \ref{sob4}.
\end{proof}

\begin{theorem}\label{th4}
Assume that \textnormal{(H1)--(H4)} are satisfied.  Then there exists a solution $(Y,M,K)$ of \textnormal{{\underline{R}BSDE}}$^T(\xi,f+dV,L)$ such that $Y$ is of class \textnormal{(D)}.
\end{theorem}
\begin{proof}
We consider a strictly positive  function $g:[0,\infty)\rightarrow\mathbb{R}$ such that $\int^{\infty}_0 g(r)\,dr<\infty$. For each $n\ge 1$, let
\[
f_n(r,y)=f(r,y)\vee (-n\cdot g(r)),
\]
From theorem  Theorem \ref{th3} we know that for each $n\ge1$ there exists a solution $(Y^n,M^n,K^n)$ to of \underline{R}BSDE$^T(\xi,f_n+dV,L)$ such that  $Y^n$ is of class (D). Since  $f_n\ge f_{n+1}$,   $Y^n\ge Y^{n+1}$ by Proposition \ref{sob4}. Set $Y=\inf_{n\ge 1} Y^n$. Then $Y$ is of class (D) and $\lim_{a\rightarrow\infty}Y_{T_a}=\xi$. To see this, consider a solution $(X,H)$ to BSDE$^T(\xi$,$f+dV)$  such that $X$ is of class (D). It exists by \cite[Proposition 2.7]{Kl}. By Proposition \ref{sob4},
\begin{equation}\label{th4.0}
X\le Y^n\le Y^1,
\end{equation}
so $Y$ is of class (D) and $\lim_{a\rightarrow\infty}Y_{T_a}=\xi$.
By (H3),  $f_n(r,Y^n_r)\rightarrow f(r,Y_r)$ as $n\rightarrow\infty$. Moreover, since $f_n\ge f_{n+1}$, it follows from  (H2) and (\ref{th4.0}) that
\begin{equation}\label{th4.1}
f(r,Y^1_r)\le f_n(r,Y^n_r)\le f_1(r,X_r).
\end{equation}
Set
\[
\tau_k=\inf\{t\ge 0:\int^t_0|f(r,Y^1_r)|+|f_1(r,X_r)\,dr\ge k\}\wedge T.
\]
Then $\{\tau_k\}$ is a chain $[0,T]$ and $(Y^n,M^n,K^n)$ is a solution to {\underline{R}BSDE}($Y^n_{\tau_k}$,$f_n+dV$,$L$) on $[0,\tau_k]$. Hence, by Proposition \ref{prop2}, for evry $\sigma\in\mathcal{T}_{\tau_k}$,
\begin{equation}\label{th4.3}
Y^n_{\sigma}=\esssup_{\tau\in\mathcal{T}_{\sigma,\tau_k}}E\Big(\int^{\tau}_{\sigma} f(r,Y^n_r)\,dr+\int^{\tau}_{\sigma}\,dV_r+L_{\tau}\mathbf{1}_{\{\tau<\tau_k\}}
+Y^n_{\tau_k}\mathbf{1}_{\{\tau=\tau_k\}}|\mathcal{F}_{\sigma}\Big).
\end{equation}
By the definition of $\tau_k$ and (\ref{th4.1}),
\begin{equation}\label{th4.4}
E\int^{\tau_k}_0|f_n(r,Y^n_r)-f(r,Y_r)|\,dr\rightarrow 0.
\end{equation}
By (\ref{th4.0}), (\ref{th4.3}), (\ref{th4.4}) and \cite[Lemma 3.19]{KRzS},
\[
Y_{\sigma}=\esssup_{\tau\in\mathcal{T}_{\sigma,\tau_k}}E\Big(\int^{\tau}_{\sigma} f(r,Y_r)\,dr+\int^{\tau}_{\sigma}\,dV_r+L_{\tau}\mathbf{1}_{\{\tau<\tau_k\}}
+Y_{\tau_k}\mathbf{1}_{\{\tau=\tau_k\}}|\mathcal{F}_{\sigma}\Big)
\]
for $\sigma\in\mathcal{T}_{\tau_k}$. By \cite[Lemma 3.1, Lemma 3.2]{GIOQ}, $Y$ is regulated and there exist $K^k\in\mathcal{V}^+_p$ and $M^k\in\mathcal{M}_{loc}$ such that for all $a\ge0$ and $t\in[0,\tau_k\wedge a]$,
\begin{align*}
Y_t=Y_{\tau_k\wedge a}+\int^{\tau_k\wedge a}_t f(r,Y_r)\,dr+\int^{\tau_k\wedge a}_t\,dV_r+\int^{\tau_k\wedge a}_t\,dK^k_r-\int^{\tau_k\wedge\wedge a}_t\,dM^k_r.
\end{align*}
Moreover, by \cite[Lemma 3.3]{GIOQ} and \cite[Proposition 2.34, p.~131]{EK}, for all $a\ge0$ and $t\in[0,\tau_k\wedge a]$,
\[
\int^{\tau_k\wedge a}_0(Y_{r-}-\overrightarrow{L}_r)\,dK^{k,*}_r+\sum_{0\le r<\tau_k\wedge a}(Y_r-L_r)\Delta^+K^k_r=0.
\]
Therefore $(Y,M^k,K^k)$ is a solution of {\underline{R}BSDE}$^{\tau_k}(Y_{\tau_k}$,$f+dV$,$L$). By  uniqueness, for every $a\ge0$,  $K^k_t=K^{k+1}_t$ and $a\ge 0$, $M^k_t=M^{k+1}_t$ for  $t\in[0,\tau_k\wedge a]$. Since $\{\tau_k\}$ is a chain, we can define processes $K$ and $M$ on the intervals $[0,T_a]$ by putting $K_t=K^k_t$, $M_t=M^k_t$ on $[0,\tau_k\wedge a]$. Since $\lim_{a\rightarrow\infty}Y_{T_a}=\xi$, the triple  $(Y,M,K)$ is a solution of {\underline{R}BSDE}$^T(\xi,f+dV,L)$. Uniqueness of the solution follows from Proposition \ref{sob4}.
\end{proof}

\begin{remark}
\label{rem.wlog1}
Instead of condition
\begin{equation}
\label{eq.acih0}
E\int^T_0 |f(r,0)|\,dr<\infty
\end{equation}
appearing in hypothesis  (H1) one can consider the following more general condition: there exists a process  $S$ which is a difference of two supermartingales of class (D) such that
\[
E\int^T_0 |f(r,S_r)|\,dr<\infty.
\]
However, without loss of generality one can  assume that (\ref{eq.acih0}) is satisfied.
Indeed, let $(Y,M,K)$ be  a solution of \underline{R}BSDE$^T(\xi,f+dV,L)$. Since $S$ is a difference of supermartingales, there exist processes $H\in\mathcal{M}_{loc}$ and $C\in\mathcal{V}^1_p$ such that for every $a\ge0$,
\[
S_t=S_{T_a}+\int^{T_a}_t\,dC_r-\int^{T_a}_t\,dH_r,\quad t\in[0,T_a].
\]
Let $\tilde{Y}=Y-S$, $\tilde{M}=M-S$. Then $(\tilde{Y},\tilde{M},K)$ is a solution of \underline{R}BSDE$^T(\tilde{\xi},\tilde{f}+d\tilde{V},\tilde{L})$ with $\tilde\xi=\xi-S_T$, $\tilde{f}(r,y)=f(r,y+S_r)$, $\tilde{V}=V-C$ and $\tilde{L}=L-S$. We have
\[
E\int^T_0|\tilde{f}(r,0)|\,dr=E\int^T_0|f(r,S_r)|\,dr<\infty.
\]
\end{remark}

\subsubsection{Optimal stopping problem with nonlinear $f$-expectation}
%and \cite[Theorem 9.2]{GIOQ}
%Observe that if $(Y,M,K)$ is a solution to RBSDE$^T(\xi,f,L)$, then it is a solution of RBSDE$^{T_a}(Y_{T_a},f,L)$
%for every $a\ge 0$.

Repeating step by step the proofs of the results of Section 8 in \cite{GIOQ} and \cite[Theorem 6.1]{GIOQ},  with using 
Proposition \ref{stw.fexp} and Proposition \ref{sob4}, we get the following result.

\begin{lemma}
\label{lm.ad7.1}
Assume that \mbox{\rm(H1)--(H4)} are satisfied.  For  $\alpha\in \TT$, let
\[
Y(\alpha)=\esssup_{\tau\in \TT}\EE^f_{\alpha,\tau}(L^\xi_{\tau}),
\]
where $L^\xi_t= L_t\mathbf{1}_{t<T}+\xi \mathbf{1}_{t=T}$. Then
\begin{enumerate}
\item[\rm(i)] There exists an optional process $Y$ which
aggregates the family $(Y(\alpha))_{\alpha\in \TT}$, that is $Y_\alpha=Y(\alpha)$ for every $\alpha\in\TT$.
\item[\rm(ii)] $Y$ is the smallest $\EE^f$-supermartingale majorizing $L^\xi$,
\item[\rm(iii)] If $L$ is u.s.c. from the right, then $Y$ coincides with the first component of the solution to RBSDE$^T(\xi,f,L)$.
\end{enumerate}
\end{lemma}

\begin{theorem}
\label{th.ad7.1}
Assume that \mbox{\rm(H1)--(H4)} are satisfied. Let $(Y,M,K)$ be a solution of RBSDE$^T(\xi,f,L)$ such that $Y$ is of class \mbox{\rm(D)}.
Then for every $\alpha\in\TT$,
\[
Y_\alpha=\esssup_{\tau\in\TT_\alpha}
\EE^f_{\alpha,\tau}(L_\tau\mathbf{1}_{\tau<T}+\xi\mathbf{1}_{\tau=T}).
\]
\end{theorem}
\begin{proof}
By Proposition \ref{stw.fexp}(i), $Y$ is an $\EE^f$-supermartingale. Of course, $Y\ge L^\xi$. Let $Y'$ be an $\EE^f$-supermartingale such that $Y'\ge L^\xi$ and
$(\bar Y,\bar M,\bar K)$
be a solution to RBSDE$^T(Y'_T,f,Y')$. Then, by Proposition \ref{sob4}, $\bar Y\ge Y$. On the other hand, since $Y'$
is an $\EE^f$-supermartingale, $Y'_\alpha=\esssup_{\tau\ge \alpha}\EE^f_{\alpha,\tau}(Y'_\tau)$. Therefore, by Lemma \ref{lm.ad7.1}(iii)
and Proposition \ref{sob4}, $\bar Y=Y'$ (since $Y'$ is u.s.c. from the right as an $\EE^f$-supermartingale).  Thus $Y'\le Y$, which implies that $Y$ is the smallest $\EE^f$-supermartingale majorizing $L^\xi$. This when combined with Lemma \ref{lm.ad7.1}(ii) gives the desired result.
\end{proof}

\nsubsection{Existence results for RBSDEs with two barriers and  Dynkin games}
\label{sec6}

In this section, we  prove existence results for reflected BSDEs with two optional barriers satisfying  Mokobodzki's condition:
\begin{enumerate}
\item[(H6)] There exists a special semimartingale $X$ such that $L\le X\le U$.
\end{enumerate}
As in the case of one barrier, we first prove the existence of integrable solutions to RBSDEs with two barriers  satisfying the following stronger condition:
\begin{enumerate}
\item[(H6*)] There exists a process $X$  being a difference of two supermartingales of class (D) on $[0,T]$ such that $L\le X\le U$ and  $E\int^T_0 |f(r,X_r)|\,dr<\infty$.
\end{enumerate}
However, we start with showing that each solution to RBSDE is the value function in a nonlinear Dynkin game. This representation will also be needed in the proof of the main result.

\subsubsection{Nonlinear Dynkin games}

\begin{definition}
Let $\tau\in\TT$ and  $H\in\mathcal{F}_{\tau}$. Write $H^c=\Omega\setminus H$. If $H^c\cap\{\tau=T\}=\emptyset$, then the pair $\rho=(\tau,H)$ is called a stopping system.
\end{definition}

We denote by $\mathcal{S}$ the set of all stopping systems and for  fixed stopping times $\sigma,\gamma\in\mathcal{T}$ we denote by $\mathcal{S}_{\sigma,\gamma}$ the set of stopping systems $\rho=(\tau,H)$ such that $\sigma\le\tau\le\gamma$. We put $\mathcal{S}_{\sigma}:=\mathcal{S}_{\sigma,T}$. Note that any stopping time $\tau\in\TT$ can be identified with a  stopping system $(\tau,\Omega)$. Therefore
we may write $\TT\subset \mathcal S$.

For a stopping system $\rho=(\tau,H)$ and for an optional process $X$, we set
\[
X^u_{\rho}=X_{\tau}\mathbf{1}_H+\overleftarrow{X}_{\tau}\mathbf{1}_{H^c},\qquad X^l_{\rho}=X_{\tau}\mathbf{1}_H+\underleftarrow{X}_{\tau}\mathbf{1}_{H^c}.
\]

Repeating step by step the proofs of \cite[Lemma 4.15, Lemma 4.17]{GIOQ2} with using Propositions \ref{stw.fexp} and \ref{sob4} we get the following result.

\begin{theorem}\label{th6.2}
Assume \textnormal{(H1)--(H4)}. Let $(Y,M,R)$ be a solution to  \textnormal{RBSDE}$^T(\xi,f+dV,L,U)$. Then for every $\alpha\in\mathcal{T}$,
\begin{align}
\label{eq.ad8.1.1}
\nonumber Y_{\alpha}&=\esssup_{\rho=(\tau,H)\in\mathcal{S}_{\alpha}}\essinf_{\delta =(\sigma,G)\in\mathcal{S}_{\alpha}}\mathcal{E}^f_{\alpha,\tau\wedge\sigma}
(L^u_{\rho}\mathbf{1}_{\{\tau\le\sigma<T\}}+U^l_{\delta}
\mathbf{1}_{\{\sigma<\tau\}}+\xi\mathbf{1}_{\{\tau=\sigma=T\}})
\\&=\essinf_{\delta=(\sigma,G)\in\mathcal{S}_{\alpha}}\esssup_{\rho
=(\tau,H)\in\mathcal{S}_{\alpha}} \mathcal{E}^f_{\alpha,\tau\wedge\sigma}(L^u_{\rho}\mathbf{1}_{\{\tau\le\sigma<T\}} +U^l_{\delta}\mathbf{1}_{\{\sigma<\tau\}}+\xi\mathbf{1}_{\{\tau=\sigma=T\}}) .
\end{align}
\end{theorem}

\begin{remark}
In general, (\ref{eq.ad8.1.1}) is not true if we replace stopping systems by stopping times.
The proof of (\ref{eq.ad8.1.1}) is much more simpler then the proof of the corresponding result for one barrier (Theorem \ref{th.ad7.1}).
This is due to the fact that in (\ref{eq.ad8.1.1}) we can always indicate  $\varepsilon$-optimal stopping systems
regardless on the regularity of barriers $L,U$. These $\varepsilon$-optimal stopping systems $\rho_\varepsilon=(\tau_\varepsilon,H_\varepsilon),\, \delta_\varepsilon=(\sigma_\varepsilon,G_\varepsilon)$
are given by the following formulas  (see \cite[(4.19)]{GIOQ2}),
\[
\tau_\varepsilon=\inf\{t\ge \alpha;\, Y_t\le L_t+\varepsilon\}\wedge T,\quad \sigma_\varepsilon=\inf\{t\ge \alpha;\, Y_t\ge U_t-\varepsilon\}\wedge T
\]
and
\[
H_\varepsilon=\{\omega\in\Omega;\, Y_{\tau_\varepsilon(\omega)}(\omega)\le L_{\tau_\varepsilon(\omega)}(\omega)+\varepsilon\},\quad
G_\varepsilon=\{\omega\in\Omega;\, Y_{\sigma_\varepsilon(\omega)}(\omega)\ge U_{\sigma_\varepsilon(\omega)}(\omega)+\varepsilon\}.
\]
Note also that formulas of type (\ref{eq.ad8.1.1}) for linear RBSDEs (however without using the notion of RBSDEs) were proved in \cite{A-N}.
\end{remark}

As a corollary to the above theorem we obtain a stability result for solutions to RBSDEs.
Before stating it,  we give  some remarks about the process $\overleftarrow X$.
Assume that $X$ is positive.
From the definition it follows easily  that $\overleftarrow X$ is the smallest progressively measurable process majorizing $X$
which is u.s.c. from the right. Let
\[
S_\alpha= \esssup_{\tau\ge \alpha}E(X_\tau|\FF_\alpha),\quad \alpha\in \TT.
\]
By \cite[page 417]{DM1}, $S$ is the smallest supermartingale majorizing $X$, so  $S$ is u.s.c. from the right and $X\le S$. Thus $\overleftarrow X\le S$. Therefore $\overleftarrow X$ is of class (D) and
\begin{equation}
\label{eq.ibdn1}
\|\overleftarrow X\|_1\le \|S\|_1=\|X\|_1.
\end{equation}

\begin{proposition}
Assume that the data $(\xi_i,f_i,L^i,U^i)$, $i=1,2$, satisfy \mbox{\rm(H1)--(H4)}.
Let $(Y^i,M^i,R^i)$ be a solution to RBSDE$^T(\xi_i,f_i,L^i,U^i),\, i=1,2.$ Then
\begin{align*}
\|Y^1-Y^2\|_1&\le E|\xi_1-\xi_2|+2\|L^1-L^2\|_1+2\|U^1-U^2\|_1\\
&\quad+\sup_{\rho,\delta\in\mathcal S} E\int_0^{\sigma\wedge\tau} |f_1-f_2|(r,\EE^{f_1}_{r,\tau\wedge\sigma}(Z^{\rho,\delta}))\,dr,
\end{align*}
where
\[
Z^{\rho,\delta}= L^u_{\rho}\mathbf{1}_{\{\tau\le\sigma<T\}}
+U^l_{\delta}\mathbf{1}_{\{\sigma<\tau\}}+\xi\mathbf{1}_{\{\tau=\sigma=T\}}.
\]

\end{proposition}
\begin{proof}
By Theorem \ref{th6.2}, $Y^i$, $i=1,2$, admits representation (\ref{eq.ad8.1.1}).
Observe that
\[
|L^{1,u}_\rho-L^{2,u}_\rho|\le |L^1_\tau-L^2_\tau|+\overleftarrow{|L^1-L^2|}_\tau,\quad |U^{1,l}_\delta-U^{2,l}_\delta|\le |U^1_\sigma-U^2_\sigma|+\overleftarrow{|L^1-L^2|}_\sigma.
\]
By this, (\ref{eq.ad8.1.1}) and Proposition \ref{stw.fexp},
\begin{align*}
E|Y^1_\alpha-Y^2_\alpha|&\le \sup_{\tau,\sigma\in\TT_\alpha}E\Big(|\xi_1-\xi_2|+|L^1_\tau-L^2_\tau|+\overleftarrow{|L^1-L^2|}_\tau+|U^1_\sigma-U^2_\sigma|+\overleftarrow{|L^1-L^2|}_\sigma\Big)\\&
\quad+\sup_{\rho,\delta\in\mathcal S_\alpha}E\int_\alpha^{\sigma\wedge\tau}|f_1-f_2|(r,\EE^{f_1}_{r,\tau\wedge\sigma}(Z^{\rho,\delta}))\,dr.
\end{align*}
Using now  (\ref{eq.ibdn1}) we get the  desired result.
\end{proof}

%Moreover, for every $\alpha\in\TT$, there exist  sequences $\{\alpha_n\},\,\{\beta_n\}\subset  \TT$ such that $\alpha_n>\alpha,\, \beta_n>\alpha$, and
%\[
%\limsup_{n\rightarrow \infty} X_{\alpha_n}=\overleftarrow X_\alpha,\qquad \liminf_{n\rightarrow \infty} X_{\beta_n}=\underleftarrow X_\alpha.
%\]
%From this, we deduce that $\overleftarrow X, \underleftarrow X$ are of class (D) too, and

\subsubsection{Reflected BSDEs with two optional barriers}

\begin{theorem}\label{th5}
Assume that \textnormal{(H1)--(H4), (H6*)} are satisfied. Then there exists a unique solution $(Y,M,R)$ of \textnormal{RBSDE}$^T(\xi,f+dV,L,U)$ such that $Y$ is of class \textnormal{(D)}. Moreover, $R\in \VV^1_p$, $M$ is a martingale of class \mbox{\rm(D)} on $[0,T]$, and
\begin{equation}
\label{th5.2.ab}
E\int_0^T|f(r,Y_r)|\,dr<\infty.
\end{equation}
\end{theorem}
\begin{proof}
Let $(Y^{1,0},M^{1,0})$ be a solution of BSDE$^T(\xi$,$f+dV)$ such that $Y^{1,0}$ is of class (D), and let $(Y^{2,0},Z^{2,0})=(0,0)$. Next, for each $n\ge 1$, let $(Y^{1,n},M^{1,n},K^{1,n})$ be a solution of \underline{R}BSDE$^T(\xi$,$f_n+dV$,$L+Y^{2,n-1})$ with
\[
f_n(r,y)=f(r,y-Y^{2,n-1}_r),
\]
and let $(Y^{2,n}, Z^{2,n},K^{2,n})$ be a solution of \underline{R}BSDE$^T$($0$,$0$,$Y^{1,n-1}-U$) such that $Y^{1,n},Y^{2,n}$ are of class (D). The existence of such  solutions follows from Theorem \ref{th4} (see also Remark \ref{rem.wlog1}).
We will show by induction that the sequences $(Y^{1,n})_{n\ge 0}$, $(Y^{2,n})_{n\ge 0}$ are increasing. Clearly  $Y^{1,1}\ge Y^{1,0}$ and $Y^{2,1}\ge Y^{2,0}$. Suppose that for fixed $n\in\mathbb{N}$, $Y^{1,n}\ge Y^{1,n-1}$ and $Y^{2,n}\ge Y^{2,n-1}$. Using this and (H2) we infer  that $f_{n+1}\ge f_n$ and $L+Y^{2,n}\ge L+Y^{2,n-1}$. Hence by Proposition \ref{sob4}, $Y^{1,n+1}\ge Y^{1,n}$. By a similar argument, $Y^{2,n+1}\ge Y^{2,n}$. Thus $(Y^{1,n})_{n\ge 0}$ and $(Y^{2,n})_{n\ge 0}$ are increasing.
We next show that $Y^1:=\sup_{n\ge 1}Y^{1,n}$ and  $Y^2:=\sup_{n\ge 1}Y^{2,n}$ are of class (D). By (H6), there exists a process $X$ such that $X$ is of class (D), $X\ge L$ and $E\int^T_0 |f(r,X_r)|\,dr<\infty$. There exist processes $H\in\mathcal{M}_{loc}$ and $C\in\mathcal{V}^1_p$ such that for every $a\ge0$,
\[
X_t=X_{T_a}+\int^{T_a}_t\,dC_r-\int^{T_a}_t\, dH_r,\quad t\in[0,T_a].
\]
This equation can be can be rewritten in the form
\[
X_t=X_{T_a}+\int^{T_a}_t f(r,X_r)\,dr+\int^{T_a}_t\,dV_r+\int^{T_a}_t\,dC'_r
-\int^{T_a}_t\, dH_r,
\]
where $C'_t=-\int^t_0 f(r,X_r)\,dr-V_t+C_t$. Let $(\tilde{X}^1,\tilde{H}^1)$ be a solution of BSDE$^T(X_T\vee\xi,f+dV+dC'^+)$ and $(\tilde{X}^2,\tilde{H}^2)$ be a solution of BSDE$^T(0,dC'^-)$, such that $\tilde{X}^1,\tilde{X}^2$ are of class (D). The existence of these solutions follows from Theorem \ref{th4}. Set $\tilde{X}=\tilde{X}^1-\tilde{X}^2$. Observe that $(\tilde{X}^1,\tilde{H}^1,0)$ is a solution of \underline{R}BSDE$^T(X_T\vee \xi$, $\tilde{f}+dV+dC'^+$, $L+\tilde{X}^2)$ with $\tilde{f}(r,x)=f(r,x-\tilde{X}^2_r)$ and $(\tilde{X}^2,\tilde{H}^2,0)$ is a solution of \underline{R}BSDE$^T$($0$, $dC'^-$, $\tilde{X}^1-U$). We will show by induction for each $n\in\mathbb{N}$, $\tilde{X}^1\ge Y^{1,n}$ and $\tilde{X}^2\ge Y^{2,n}$. Let us show this property by induction. Let $n=0$. Since $\tilde{X}^2\ge 0$, it follows from (H2) that $\tilde{f}\ge f$. Hence, by Proposition \ref{sob4}, $\tilde{X}^1\ge Y^{1,0}$. Since $Y^{2,0}=0$, it is clear that
$\tilde{X}^2\ge Y^{2,0}$. Suppose that for  fixed $n\in\mathbb{N}$, $\tilde{X}^1\ge Y^{1,n}$ and $\tilde{X}^2\ge Y^{2,n}$. Using this and (H2) we conclude that $\tilde{f}\ge f_{n+1}$, $L+\tilde{X}^2\ge L+Y^{2,n}$ and $\tilde{X}-U\ge Y^{1,n}-U$. Hence, by Proposition \ref{sob4}, $\tilde{X}^2\ge Y^{1,n+1}$, $\tilde{X}^2\ge Y^{2,n+1}$. Therefore,  for each $n\in\mathbb{N}$, $\tilde{X}^1\ge Y^{1,n}$ and $\tilde{X}^2\ge Y^{2,n}$. We have
\begin{equation}
\label{th5.1}
Y^{1,0}\le Y^{1,n}\le \tilde{X}^1,\quad Y^{2,0}\le Y^{2,n}\le \tilde{X}^2,
\end{equation}
so $Y^1,Y^2$ are of class (D) and
\begin{equation}\label{th5.1.1.1}
\xi\le\lim_{a\rightarrow\infty}Y^1_{T_a}\quad\lim_{a\rightarrow\infty}Y^2_{T_a}=0.
\end{equation}
We will show that there exist $M^1,M^2\in\mathcal{M}_{loc}$ , $K^1,K^2\in\mathcal{V}^+$ such that $(Y^1,M^1,K^1)$ is a solution of \underline{R}BSDE$^T$($\xi$,$\hat{f}+dV$,$L+Y^2$) with
\[
\hat{f}(r,y)=f(r,y-Y^2_r)
\]
and $(Y^2,M^2,K^2)$ is a solution of \underline{R}BSDE$^T$($0$,$0$,$Y^1-U$). By (H3),
\begin{equation}\label{th5.1.1}
f(r,Y^{1,n}_r-Y^{2,n-1}_r)\rightarrow f(r,Y^1_r-Y^2_r)
\end{equation}
as  $n\rightarrow\infty$. Moreover, by (H2) and (\ref{th5.1}),
\begin{equation}
\label{th5.2}
f(r,X^1_r)\le f(r,Y^{1,n}_r-Y^{2,n-1}_r)\le f(r,Y^{1,0}_r-X^2_r).
\end{equation}
Set
\[
\tau_k=\inf\{t\ge 0:\int^t_0|f(r,Y^{1,0}_r-X^2_r)|+|f(r,X^1_r)|\,dr\ge k\}\wedge T.
\]
Then $\{\tau_k\}$ is a chain on $[0,T]$ and  $(Y^{1,n},Z^{1,n},K^{1,n})$ is a solution of \underline{R}BSDE$^{\tau_k}(Y^{1,n}_{\tau_k}$,$f_n+dV$,$L+Y^{2,n-1}$). Hence,  by Proposition \ref{prop2}, for every $\sigma\in\mathcal{T}_{\tau_k}$ we have
\begin{align}
\label{th5.3}
Y^{1,n}_{\sigma}&=\esssup_{\tau\in\mathcal{T}_{\sigma,\tau_k}}E\Big(\int^{\tau}_{\sigma} f(r,Y^{1,n}_r-Y^{2,n-1}_r)\,dr+\int^{\tau}_{\sigma}\,dV_r \nonumber \\
&\quad +(L_{\tau}+Y^{2,n-1}_{\tau})\mathbf{1}_{\{\tau<\tau_k\}}
+Y^{1,n}_{\tau_k}\mathbf{1}_{\{\tau=\tau_k\}}|\mathcal{F}_{\sigma}\Big).
\end{align}
By the definition of $\tau_k$, (\ref{th5.1.1}) and (\ref{th5.2}),
\begin{equation}\label{th5.4}
E\int^{\tau_k}_0|f(r,Y^{1,n}_r-Y^{2,n-1}_r)-f(r,Y^1_r-Y^2_r)|\,dr\rightarrow 0
\end{equation}
as $n\rightarrow\infty$. By (\ref{th5.1}), (\ref{th5.3}), (\ref{th5.4}) and \cite[Lemma 3.19]{KRzS},
\begin{align*}
Y^1_{\sigma}&=\esssup_{\tau\in\mathcal{T}_{\sigma,\tau_k}}E\Big(\int^{\tau}_{\sigma} f(r,Y^1_r-Y^2_r)\,dr+\int^{\tau}_{\sigma}\,dV_r \\
&\quad+(L_{\tau}+Y^2_{\tau})\mathbf{1}_{\{\tau<\tau_k\}}
+Y^1_{\tau_k}\mathbf{1}_{\{\tau=\tau_k\}}|\mathcal{F}_{\sigma}\Big).
\end{align*}
for every $\sigma\in\mathcal{T}_{\tau_k}$. By \cite[Lemma 3.8]{Kl},
$\lim_{a\rightarrow\infty}Y^1_{\tau_k\wedge a}\le Y^1_{\tau_k}$,
and since $\{\tau_k\}$ is a chain,
$\lim_{a\rightarrow\infty}Y^1_{T_a}\le \xi$. By this and
(\ref{th5.1.1.1}),
\[
\lim_{a\rightarrow\infty}Y_{T_a}= \xi.
\]
As for $Y^2$, by Proposition \ref{prop2}, for all $n\ge1$ and  $\sigma\in\mathcal{T}_T$,
\[
Y^{2,n}_{\sigma}=\esssup_{\tau\in \mathcal{T}_{\sigma}}E\Big((Y^{1,n-1}_{\tau}-U_{\tau})
\mathbf{1}_{\{\tau<T\}}|\mathcal{F}_{\sigma}\Big).
\]
Letting $n\rightarrow\infty$ and using (\ref{th5.1}) we get
\[
Y^2_{\sigma}=\esssup_{\tau\in \mathcal{T}_{\sigma}}E\Big((Y^1_{\tau}-U_{\tau})
\mathbf{1}_{\{\tau<T\}}|\mathcal{F}_{\sigma}\Big).
\]
By \cite[Lemma 3.1, Lemma 3.2]{GIOQ}, $Y^1$ is regulated and there exist $K^k\in\mathcal{V}^+_p$ and $M^k\in\mathcal{M}_{loc}$ such that for every $a\ge0$,
\[
Y^1_t=Y^1_{\tau_k\wedge a}+\int^{\tau_k\wedge a}_t f(r,Y^1_r-Y^2_r)\,dr+\int^{\tau_k\wedge a}_t\,dV_r+\int^{\tau_k\wedge a}_t\,dK^{1,k}_r-\int^{\tau_k\wedge a}_t \,dM^k_r
\]
for $t\in[0,\tau_k\wedge a]$. Furthermore,  by \cite[Lemma 3.3]{GIOQ} and \cite[Proposition 2.34, p.~131]{EK},
\[
\int^{\tau_k\wedge a}_0(Y^1_{r-}-\overrightarrow{L_r+Y^2_r})\,dK^{1,k,*}_r
+\sum_{0\le r<\tau_k\wedge a}(Y^1_r-(L_r+Y^2))\Delta^+K^{1,k}_r=0
\]
for $a\ge0$. Therefore $(Y^1,Z^{1,k},K^{1,k})$ is a solution of \underline{R}BSDE$^{\tau_k}(Y^1_{\tau_k},\hat{f}+dV,L+Y^2$). By uniqueness, for every $a\ge$, $K^{1,k}_t=K^{1,k+1}_t$, $M^{1,k}_t=M^{1,k+1}_t,\,t\in[0,\tau_k\wedge a]$. Therefore, since $\{\tau_k\}$ is a chain, we can define processes $M^1$ and $K^1$ on the intervals $[0,T_a]$ by putting $M^1_t=M^{1,k}_t$, $K^1_t=K^{1,k}_t$, $t\in[0,\tau_k\wedge a]$, $a\ge 0$. From this and the fact that $\lim_{a\rightarrow\infty}Y^1_{T_a}=\xi$ it follows that $(Y^1,Z^1,K^1)$ is a solution of \underline{R}BSDE$^T(\xi,\hat{f}+dV,L+Y^2)$. By \cite[Lemma 3.1, Lemma 3.2]{GIOQ} again, $Y^2$ is regulated and  there exist $K^2\in\mathcal{V}^+_p,M^2\in\mathcal{M}_{loc}$ such that for every $a\ge0$,
\[
Y^2_t=\int^{T_a}_t\,dK^2_r-\int^{T_a}_t \,dM^2_r,\quad t\in[0,T_a].
\]
By \cite[Lemma 3.3]{GIOQ} and \cite[Proposition 2.34, p.~131]{EK}, for every $a\ge0$,
\[
\int^{T_a}_0(Y^2_{r-}-\overrightarrow{Y^1_r-U_r})\,dK^{2,*}_r
+\sum_{0\le r<T_a}(Y^2_r-(Y^1_s-U_s))\Delta^+K^2_r=0.
\]
Since $\lim_{a\rightarrow\infty}Y^2_{T_a}=0$, we see  that $(Y^2,M^2,K^2)$ is a solution of \underline{R}BSDE$^T(0,0,Y^1-U)$.
Write $Y=Y^1-Y^2$, $M=M^1-M^2$, $R=K^1-K^2$. We shall show that $(Y,Z,R)$ is a solution of RBSDE$^T$($\xi$,$f+dV$,$L$,$U$). Observe that for every $a\ge0$,
\[
Y_t=Y_{T_a}+\int^{T_a}_t f(r,Y_r)\,dr+\int^{T_a}_t\,dV_r+\int^{T_a}_t\,dR_r-\int^{T_a}_t \,dM_r,\quad t\in[0,T_a],
\]
and $\lim_{a\rightarrow\infty}Y_{T_a}=\xi$. Clearly $L\le Y\le U$. Furthermore, $R$ satisfies the minimality condition because for every $a\ge0$ we have
\begin{align*}
&\int^{T_a}_0 (Y_{r-}-\overrightarrow{L}_r)\,dR^{+,*}_r+\sum_{0\le r<T_a}(Y_r-L_r)\Delta^+R^+_r\\
&\quad\le \int^{T_a}_0 (Y^1_{r-}-\overrightarrow{L_r+Y^2_r})\,dK^{1,*}_r+\sum_{0\le r<T_a}(Y^1_r-(L_r+Y^2_r))\Delta^+K^1_r=0
\end{align*}
and
\begin{align*}
&\int^{T_a}_0 (\underrightarrow{U}_r-Y_{r-})\,dR^{-,*}_r+\sum_{0\le r<T_a}(U_r-Y_r)\Delta^+R^-_r\\
&\quad\le \int^{T_a}_0 (Y^2_{r-}-\overrightarrow{Y^1_r-U_r})\,dK^{2,*}_r+\sum_{0\le r<T_a}(Y^2_r-(Y^1_r-U_r))\Delta^+K^2_r=0,
\end{align*}
Since $Y^1,Y^2$ are of class (D),  $Y$ is of class (D), too.  Uniqueness of the solution follows from Proposition \ref{sob4}.
Finally, by Proposition \ref{sob4},
$\overline Y\le Y\le \underline Y$,
where $\underline Y$ (resp. $\overline Y$) is a solution to \underline{R}BSDE$^T(\xi,f+dV,L$) (resp. $\overline{\rm{R}}$BSDE$^T(\xi,f+dV,U)$). Therefore (\ref{th5.2.ab}) follows from (H2) and Proposition \ref{prop3}.
\end{proof}

\begin{proposition}\label{prop6.1}
Assume that $f,f_n$, $n\ge 1$, satisfy \textnormal{(H1)--(H4)} and $f_n\nearrow f$ as $n\rightarrow\infty$. Let $\{L^n\}$
be a sequence of optional processes of class \mbox{\rm(D)} on $[0,T]$ such that $L^n\nearrow L$.
\begin{enumerate}
\item[\rm(i)] Let $(Y,M,K)$ (resp. $(Y^n,M^n,K^n)$) be a solution to $\underline{\rm{R}}$\textnormal{BSDE}$^T(\xi,f,L)$ (resp. $\underline{\rm{R}}$\textnormal{BSDE}$^T(\xi,f,L_n)$). Then $Y^n\nearrow Y$.

\item[\rm(ii)]
Let $(Y,M,R)$ (resp. $(Y^n,M^n,R^n)$) be a solution of \textnormal{RBSDE}$^T(\xi,f+dV,L,U)$ (resp. \textnormal{RBSDE}$^T(\xi,f_n+dV,L,U)$) such that $Y$ (resp. $Y^n$) is of class \textnormal{(D)}. Then $\|Y-Y^n\|_1\rightarrow 0$ as $n\rightarrow\infty$.
\end{enumerate}
\end{proposition}
\begin{proof}
(i) By Theorem \ref{th.ad7.1} and Proposition \ref{stw.fexp}(ii), for every $\alpha\in\TT$,
\[
Y^n_\alpha=\esssup_{\tau\in\TT_\alpha}\EE^f_{\alpha,\tau}(L^n_\tau\mathbf{1}_{\tau<T}
+\xi\mathbf{1}_{\tau=T})\le
\esssup_{\tau\in\TT_\alpha}\EE^f_{\alpha,\tau}
(L_\tau\mathbf{1}_{\tau<T}+\xi\mathbf{1}_{\tau=T})=Y_\alpha.
\]
By Proposition \ref{stw.fexp}(ii), $Y^n\le Y^{n+1},\, n\ge 1$. Set $X:=\sup_{n\ge 1}Y^n$. Then $X\le Y$. 
For the opposite inequality first observe that for all $\alpha\in\TT$ and $\tau\in\TT_\alpha$,
\[
X_\alpha\ge \esssup_{\tau\in\TT_\alpha}\EE^f_{\alpha,\tau}(L^n_\tau\mathbf{1}_{\tau<T}
+\xi\mathbf{1}_{\tau=T})\ge \EE^f_{\alpha,\tau}(L^n_\tau\mathbf{1}_{\tau<T}+\xi\mathbf{1}_{\tau=T}),\quad n\ge 1.
\]
By Proposition \ref{stw.fexp}(iii), $\EE^f_{\alpha,\tau}(L^n_\tau\mathbf{1}_{\tau<T}+\xi\mathbf{1}_{\tau=T})\rightarrow \EE^f_{\alpha,\tau}(L_\tau\mathbf{1}_{\tau<T}+\xi\mathbf{1}_{\tau=T})$ a.s. Hence
\[
X_\alpha\ge \EE^f_{\alpha,\tau}(L_\tau\mathbf{1}_{\tau<T}+\xi\mathbf{1}_{\tau=T})
\]
for all $\alpha\in\TT$ and $\tau\in\TT_\alpha$. Thus
\[
X_\alpha\ge \esssup_{\tau\in\TT_\alpha}\EE^f_{\alpha,\tau}(L_\tau\mathbf{1}_{\tau<T}
+\xi\mathbf{1}_{\tau=T})=Y_\alpha,\quad \alpha\in\TT.
\]
(ii) For each $n\ge 1$, $f_n\le f_{n+1}$, so by Proposition \ref{sob4}, $Y^n\le Y^{n+1}$. Set $\tilde{Y}:=\sup_{n\ge 1}Y^n$. By Proposition \ref{sob4}, for each $n\ge 1$,
\begin{equation}\label{prop6.1.0}
Y^1\le Y^n\le Y,
\end{equation}
so $\tilde{Y}$ is of class (D) and $\lim_{a\rightarrow\infty}\tilde{Y}_{T_a}=\xi$.
We will show that $Y=\tilde{Y}$. By (H2),
\begin{equation}\label{prop6.1.1}
\sgn(Y_r-Y^n_r)(f(r,Y_r)-f_n(r,Y^n_r))\le|f(r,Y_r)-f_n(r,Y_r)|.
\end{equation}
By the minimality conditions for $R$ and $R^n$,
\begin{equation}\label{prop6.1.1.1}
\sgn(Y_{r-}-Y^n_{r-})d(R^1_r-R^2_r)^*
\le\mathbf{1}_{\{Y_{r-}>Y^n_{r-}\}}\,dR^{*,+}+\mathbf{1}_{\{Y_{r-}>Y^n_{r-}\}}\,dR^{n,-,*}=0
\end{equation}
and
\begin{equation}\label{prop6.1.1.2}
\sgn(Y_r-Y^n_r)\Delta^+(R_r-R^n_r)\le\mathbf{1}_{\{Y_r>Y^n_r\}}\Delta^+R_r^++\mathbf{1}_{\{Y_r>Y^n_r\}}\Delta^+R_r^{n,-}=0.
\end{equation}
By \cite[Corollary A.5]{KRzS}, for all $a\ge0$ and stopping times $\sigma,\tau\in\mathcal{T}_{T_a}$ such that $\sigma\le\tau$ we have
\begin{align*}
|Y_{\sigma}-Y^n_{\sigma}|&\le|Y_{\tau}-Y^n_{\tau}|+\int^{\tau}_{\sigma}
\sgn(Y_r-Y^n_r)(f(r,Y_r)-f_n(r,Y^n_r))\,dr\\
&\quad+\int^{\tau}_{\sigma}\sgn(Y_{r-}-Y^n_{r-})\,d(R^1_r-R^2_r)^*+\sum_{\sigma\le r<\tau}\sgn(Y_r-Y^n_r)\Delta^+(R^1_r-R^2_r)\\
&\quad-\int^{\tau}_{\sigma}\sgn(Y_{r-}-Y^n_{r-})\,d(M_r-M^n_r).
\end{align*}
By the above inequality and (\ref{prop6.1.1})-(\ref{prop6.1.1.2}),
\begin{align}\label{prop6.1.3}
|Y_{\sigma}-Y^n_{\sigma}|&\le|Y_{\tau}-Y^n_{\tau}|
+\int^{\tau}_{\sigma}|f(r,Y_r)-f_n(r,Y_r)|\,dr\nonumber\\
&\quad-\int^{\tau}_{\sigma}\sgn(Y_{r-}-Y^n_{r-})\,d(M_r-M^n_r).
\end{align}
By (H2), (\ref{prop6.1.0}) and the fact that  $f_n\le f_{n+1}$ and $f_n\le f$ we have
\begin{equation}\label{prop6.1.4}
f_1(r,Y_r)\le f_n(r,Y_r)\le f(r,Y^1_r).
\end{equation}
Let
\[
\sigma_k=\inf\{t\ge 0:\int^t_0|f_1(r,Y_r)|+|f(r,Y^1_r)|\,dr\ge k\}\wedge T.
\]
Observe that $\{\sigma_k\}$ is a chain on $[0,T]$. By the definition of $\{\sigma_k\}$ and (\ref{prop6.1.4}),
\begin{equation}\label{prop6.1.5}
E\int^{\sigma_k}_0|f(r,Y_r)-f_n(r,Y_r)|\,dr\rightarrow 0
\end{equation}
as $n\rightarrow\infty$. Let $\{\gamma_k\}$ be a localizing sequence on $[0,T_a]$ for the local martingale $\int^{\cdot}_0\sgn(Y_{r-}-Y^n_{r-})\,d(M_r-M^n_r)$. By (\ref{prop6.1.3}) with $\tau$ replaced by $\tau_k=\sigma_k\wedge\gamma_k$, $\tau_k\ge\sigma$ we have
\begin{align}\label{prop6.1.6}
|Y_{\sigma}-Y^n_{\sigma}|&\le|Y_{\tau_k}-Y^n_{\tau_k}|
+\int^{\tau_k}_{\sigma}|f(r,Y_r)-f_n(r,Y_r)|\,dr \nonumber\\
&\quad-\int^{\tau_k}_{\sigma}\sgn(Y_{r-}-Y^n_{r-})\,d(M_r-M^n_r).
\end{align}
Since $Y$ and $Y^n$ are of class (D) and we know that (\ref{prop6.1.5}) holds true,
taking the expectation in (\ref{prop6.1.6}),   letting $n\rightarrow\infty$ and then $k\rightarrow\infty$,  we obtain $E|Y_{\sigma}-\tilde{Y}_{\sigma}|\le E|Y_{T_a}-\tilde{Y}_{T_a}|$ for  $a\ge 0$. Letting now $a\rightarrow\infty$ we get $E|Y_{\sigma}-\tilde{Y}_{\sigma}|=0$. Therefore, by the Section Theorem (see, e.g., \cite[Chapter IV, Theorem 86]{dm}), $|Y_t-\tilde{Y}_t|=0$, $t\in[0,T]$, so $Y=\tilde{Y}$.
\end{proof}

\begin{theorem}\label{th5.5.5.5}
Assume that \textnormal{(H1)--(H4), (H6)} are satisfied. Then there exists a unique solution $(Y,M,R)$ to \textnormal{RBSDE}$^T(\xi,f+dV,L,U)$ such that $Y$ is of class \textnormal{(D)}.
\end{theorem}
\begin{proof}
Let $X$ be the process appearing in condition (H6). Since $X$ is a special semimartingale, there exists an increasing sequence $\{\gamma_k\}\subset \TT$ such that $X$ is a difference of supermartingales of class (D) on $[0,\gamma_k]$ for every $k\ge1$. Let $\varrho$ be a strcitly positive Borel measurable function on $\BR^+$ such that $\int_0^\infty\varrho(t)\,dt<\infty$, and let
\[
f_{n,m}(t,y)= \frac{n\varrho(t)}{1+n\varrho(t)}\max\{\min\{f(t,y),n\},-m\}.
\]
Observe that $f_{n,m}$ is increasing with respect to $n$ and decreasing with respect to $m$. Moreover, $f_{n,m}(t,y)\nearrow f_m(t,y)=\max\{f(t,y),-m\}$ as $n\rightarrow\infty$  and $f_m(t,y)\searrow f(t,y)$ as $m\rightarrow\infty$. Let $\hat L,\hat U$ be regulated processes defined by 
\[
\hat L_\alpha=\essinf_{\tau\in\TT_\alpha}E(L_\tau|\FF_\alpha),\quad \hat U_\alpha=\esssup_{\tau\in\TT_\alpha}E(U_\tau|\FF_\alpha),\quad \alpha\in\TT.
\]
By Lemma \ref{lm.ad5.1}, $-\hat L, \hat U$  are supermartingales of class (D) on $[0,T]$. It is clear that $\hat L\le L\le U\le \hat U$.
We define
\[
L^n_t= L_t\mathbf{1}_{\{t\le \gamma_n\}}+\hat L_t\mathbf{1}_{\{t>\gamma_n\}},
\qquad U^n_t= U_t\mathbf{1}_{\{t\le \gamma_n\}}
+\hat U_t\mathbf{1}_{\{t>\gamma_n\}}.
\]
Observe that
\[
\hat L\le L^n\le L^{n+1}\le L\le U\le U^{n+1}\le U^n\le \hat U,\quad n\ge 1.
\]
Moreover, $L^n\nearrow L$ and $U^n\searrow U$. Finally, we set
\[
X^{n,m}_t= X_t\mathbf{1}_{\{t\le\gamma_n\wedge \gamma_m\}}
+\hat L_t\mathbf{1}_{\{t>\gamma_n\ge \gamma_m\}}
+\hat U_t\mathbf{1}_{\{t>\gamma_m> \gamma_n\}}.
\]
Observe that $L^n\le X^{n,m}\le U^m$ and that the process $X^{n,m}$ is a difference of supermartingales of class (D) on $[0,T]$.
Therefore, by Theorem  \ref{th5}, there exists a unique solution $(Y^{n,m},M^{n,m},R^{n,m})$ of
RBSDE$^T(\xi,f_{n,m},L^n,U^m)$ such that $Y^{n,m}$ is of class (D). By Proposition \ref{sob4}, $\{Y^{n,m}\}$
is nondecreasing  with respect to $n$ and nonincreasing  with respect to $n$. Set
\[
Y^m= \sup_{n\ge 1} Y^{n,m},\qquad Y=\inf_{m\ge 1} Y^m.
\]
Since $\hat L\le Y^{n,m}\le\hat U,\, n,m\ge 1$ and $L,U$ are of class (D), by the diagonal method we can find a subsequence $(m_n)$ of $(m)$ such that $\lim_{n\rightarrow\infty}E|Y^{n,m_n}_{\gamma_{k}}-Y_{\gamma_k}|=0$ for every $k\ge 1$.
Let $k\le n\wedge m$.
By Theorem \ref{th6.2}, for every $\alpha\in\TT_{0,\gamma_k}$,
\[
Y^{n,m}_\alpha=\esssup_{\rho=(\tau,H)\in\mathcal{S}_{\alpha,\gamma_k}}
\essinf_{\delta=(\sigma,G)\in\mathcal{S}_{\alpha,\gamma_k}}
\mathcal{E}^{f_{n,m}}_{\alpha,\tau\wedge\sigma}
(L^u_{\rho}\mathbf{1}_{\{\tau\le\sigma<\gamma_k\}}
+U^l_{\delta}\mathbf{1}_{\{\sigma<\tau\}}
+Y^{n,m}_{\gamma_k}\mathbf{1}_{\{\tau=\sigma=\gamma_k\}}).
\]
Set, for every $\alpha\in\TT_{0,\gamma_k}$,
\[
Y(\alpha)=\esssup_{\rho=(\tau,H)\in\mathcal{S}_{\alpha,\gamma_k}}
\essinf_{\delta=(\sigma,G)\in\mathcal{S}_{\alpha,\gamma_k}}
\mathcal{E}^{f}_{\alpha,\tau\wedge\sigma}
(L^u_{\rho}\mathbf{1}_{\{\tau\le\sigma<\gamma_k\}}
+U^l_{\delta}\mathbf{1}_{\{\sigma<\tau\}}
+Y_{\gamma_k}\mathbf{1}_{\{\tau=\sigma=\gamma_k\}}).
\]
By Proposition \ref{stw.fexp},
\begin{align*}
E|Y^{n,m_n}_\alpha-Y(\alpha)|&\le E|Y^{n,m_n}_{\gamma_k}-Y_{\gamma_k}|\\&\quad
+\sup_{\rho,\delta\in\mathcal S_{\alpha,\gamma_k}}E\int_\alpha^{\gamma_k}
|f-f_{n,m_n}|(r,\EE^f_{r,\tau\wedge\sigma}(Z^{\rho,\delta,\gamma_k}))\,dr,
\end{align*}
where
\[
Z^{\rho,\delta,\gamma_k}
=L^u_{\rho}\mathbf{1}_{\{\tau\le\sigma<\gamma_k\}}
+U^l_{\delta}\mathbf{1}_{\{\sigma<\tau\}}
+Y_{\gamma_k}\mathbf{1}_{\{\tau=\sigma=\gamma_k\}}.
\]
Observe that
\[
|Z^{\rho,\delta,\gamma_k}|\le |L_{\tau\wedge\sigma}|+|\overleftarrow L_{\tau\wedge\sigma}|+|U_{\tau\wedge\sigma}|+|\underleftarrow U_{\tau\wedge \sigma}|+|Y_{\tau\wedge\sigma}|.
\]
Since all the processes $|L|, |\overleftarrow L|, |U|, |\underleftarrow U|, |Y|$ are of class (D) on $[0,T]$, using Lemma \ref{lm.ad3.ilace} we conclude  that
there exists a supermartingale $U$ of class (D) on $[0,T]$ such that
\begin{equation}
\label{eq.ad9.l1}
\EE^f_{\alpha,\tau\wedge\sigma}(Z^{\rho,\delta,\gamma_k})\le U_\alpha,\qquad \alpha\in \TT_{0,\tau\wedge\sigma}.
\end{equation}
Since $U$ is of class (D) on $[0,T]$, by using (H4) we see that there exists a chain $\{\tau_k\}$ such that
\begin{equation}
\label{eq.ad9.l2}
E\int_0^{\tau_k}|f(r,U_r)|\,dr\le k,\quad k\ge 1.
\end{equation}
By replacing $\gamma_k$ by $\gamma_k\wedge\tau_k$ we may assume that $\gamma_k\le \tau_k$.
By (\ref{eq.ad9.l1}), (\ref{eq.ad9.l2}), (H2), (H3) and the Lebesgue dominated convergence theorem,
\[
\sup_{\rho,\delta\in\mathcal S_{\alpha,\gamma_k}}E\int_\alpha^{\gamma_k}|f-f_{n,m_n}|(r,\EE^f_{r,\tau\wedge\sigma}(Z^{\rho,\delta,\gamma_k}))\,dr\rightarrow 0
\]
as $n\rightarrow\infty$. Thus $Y_\alpha=Y(\alpha),\, \alpha\in\TT$. By Theorem \ref{th6.2}, $Y$ is the first component of the solution to RBSDE$^{\gamma_k}(Y_{\gamma_k},f,L,U)$. What is left is to show that $Y_{T_a}\rightarrow \infty$ as $a\rightarrow \infty$. For this we observe that by Proposition \ref{sob4},
\[
\overline Y^{n,m}\le Y^{n,m}\le \underline{Y}^{n,m},\quad n,m\ge 1,
\]
where $(\underline Y^{n,m},\underline M^{n,m},\underline K^{n,m})$ is a solution to $\mbox{\underline R}$BSDE$^T(\xi,f_m,L^n)$
and $(\overline Y^{n,m},\overline M^{n,m},\overline K^{n,m})$ is a solution to $\overline{\rm{R}}$BSDE$^T(\xi,f_n,U^m)$.
By Proposition \ref{prop6.1}(ii), $\overline Y^{n,m}\nearrow \overline Y^m$, and by Proposition \ref{prop6.1}(i), $\underline Y^{n,m}\nearrow \underline Y^m$, where $(\overline Y^m,\overline M^m,\overline K^m)$ is a solution to $\overline{\rm{R}}$BSDE$^T(\xi,f,U^m)$,
and $(\underline Y^m,\underline M^m,\underline K^m)$ is a solution to $\mbox{\underline R}$BSDE$^T(\xi,f_m,L)$. Therefore
$\overline Y^{m}\le Y^{m}\le \underline{Y}^{m}$.
Letting $m\rightarrow \infty$ and using once again  Proposition \ref{prop6.1} yields
\[
\overline Y\le Y\le \underline{Y},
\]
where $(\overline Y,\overline M,\overline K)$ is a solution to $\overline{\rm{R}}$BSDE$^T(\xi,f,U)$,
and $(\underline Y,\underline M,\underline K)$ is a solution to $\mbox{\underline R}$BSDE$^T(\xi,f,L)$.
From this we get the desired result.
\end{proof}

\begin{theorem}\label{th5.5.5.5v1}
Assume that \textnormal{(H2)--(H4), (H6)} are satisfied and  that $E|\xi|<\infty$ and $|V|_T<\infty$ a.s. Then there exists a unique solution $(Y,M,R)$ to \textnormal{RBSDE}$^T(\xi,f+dV,L,U)$ such that $Y$ is of class \textnormal{(D)}.
\end{theorem}
\begin{proof}
By (H4) and the assumption that $|V|_T<\infty$ a.s., there exists a chain $\{\tau_k\}$ on $[0,T]$ such that
\[
E\int_0^{\tau_k}|f(r,0)|\,dr+E\int_0^{\tau_k}\,d|V|_r<\infty,\quad k\ge 1.
\]
Therefore repeating the proof of Theorem \ref{th5.5.5.5} with $\gamma_k$ replaced by $\tau_k$ we get the existence of a regulated process $Y$ of class (D) on $[0,T]$ such that $Y$ is the first component of the solution to RBSDE$^T(Y_{\tau_k},f,L,U)$ for every $k\ge 1$. It remains to show that $Y_{T_a}\rightarrow \xi$ as $a\rightarrow \infty$. We can not argue as in the proof of Theorem \ref{th5.5.5.5}, because  in general, under the assumptions of our theorem, there are no  solutions to $\overline{\rm{R}}$BSDE$^T(\xi,f,U)$
and $\mbox{\underline R}$BSDE$^T(\xi,f,L)$. Instead, to show that $Y_{T_a}\rightarrow\xi$ we use the fact that $\{\tau_k\}$ is a chain. By the definition of a solution to RBSDE$^T(Y_{\tau_k},f,L,U)$, $Y_{\tau_k\wedge a}\rightarrow Y_{\tau_k}$ as $a\rightarrow \infty$ for every $k\ge 1$. Since $\{\tau_k\}$
is a chain, $Y_{\tau_k(\omega)}(\omega)=Y_{T(\omega)}(\omega)=\xi(\omega),\, k\ge k_\omega$, which implies the desired convergence.
\end{proof}

\subsection*{Acknowledgements}
{\small This work was supported by Polish National Science Centre
(Grant No. 2016/23/B/ST1/01543).}

\end{document}